\documentclass[a4paper,11pt]{article}

\usepackage{amsmath}
\usepackage{amsfonts}
\usepackage{amssymb}
\usepackage{german, ngerman}
\usepackage[german]{babel}
\usepackage[latin1]{inputenc}
\usepackage{graphics}
\usepackage{verbatim}
\usepackage{listings}
\usepackage{dsfont}
\usepackage{graphicx}
\usepackage{geometry}
\usepackage{url}
\usepackage{bm}
\usepackage{framed}
\usepackage{color}
\usepackage{textcomp}

\newtheorem{set2}{Satz}[section]
\newtheorem{theorem}[set2]{Theorem}

\newtheorem{definition}[set2]{Definition}

\newtheorem{lemma}[set2]{Lemma}
\newtheorem{notation[set2]}{Notation}

\newtheorem{proposition}[set2]{Proposition}
\newtheorem{remark}[set2]{Remark}

\newcommand{\ep}{\hfill{$\square$}}
\newenvironment{proof}[1][Proof]{\textit{#1.} }{\
\\}

\def\XXint#1#2#3{{\setbox0=\hbox{$#1{#2#3}{\int}$}
\vcenter{\hbox{$#2#3$}}\kern-.5\wd0}}

\newcommand{\Hb}{H^2(\Omega;\R^n)}
\newcommand{\Wa}{W^{1,p}(\Omega)}

\newcommand{\bl}{\left(}
\newcommand{\br}{\right)}
\newcommand{\e}{\varepsilon}
	
\newcommand{\di}{\,\mathrm{div}}
\newcommand{\R}{\mathbb{R}}
\newcommand{\N}{\mathbb{N}}
\newcommand{\C}{\mathcal}
\newcommand{\ol}{\overline}
\newcommand{\dx}{\,\mathrm dx}

\newcommand{\dt}{\,\mathrm dt}
\newcommand{\ds}{\,\mathrm ds}
\newcommand{\diota}{\,\mathrm d\iota}

\newcommand{\dxt}{\,\mathrm dx\,\mathrm dt}
\newcommand{\dxs}{\,\mathrm dx\,\mathrm ds}

\newcommand{\GammaD}{{\Gamma_\mathrm{D}}}

\newcommand{\CC}{\mathbf{C}}
\newcommand{\Rn}{\mathbb{R}^n}
\newcommand{\Rnn}{\mathbb{R}^{n\times n}}

\allowdisplaybreaks

\begin{document}
\selectlanguage{english}
\phantom{w} \vspace{1mm}
\begin{center}
\LARGE
Existence of weak solutions for a hyperbolic-parabolic phase field system with mixed boundary conditions on non-smooth domains\footnote{This project is supported by the Research Center  
        ``Mathematics for Key Technologies''  {\sc Matheon} in Berlin (Germany).}
\\[4mm]

\end{center}
\begin{center}
\vspace{2mm}
	Christian Heinemann$^2$, Christiane Kraus\footnote{Weierstrass Institute for Applied Analysis and Stochastics (WIAS), Mohrenstr. 39, 10117 Berlin
	\\
	E-mail: \texttt{christian.heinemann@wias-berlin.de} and
\texttt{christiane.kraus@wias-berlin.de}}
\\\vspace{2mm}
	December 10, 2013
\end{center}
\noindent

\begin{abstract}
The aim of this paper is to prove existence of weak solutions of
hyperbolic-parabolic evolution inclusions defined on Lipschitz domains with mixed boundary conditions describing,
for instance, damage processes and elasticity with inertia terms.
To this end, a suitable weak formulation to deal with such evolution inclusions in a non-smooth setting is presented.
Then, existence of weak solutions is proven by utilizing time-discretization, $H^2$-regularization of the displacement variable and variational techniques
from \cite{WIAS1520} to recover the subgradients after the limit passages.
\end{abstract}

\\[2mm]
{\it Keywords:} {Hyperbolic-parabolic systems, doubly nonlinear
  differential inclusions, existence results, 
  energetic solutions, weak solutions,  linear elasticity, rate-dependent damage systems.  \\[2mm]
}


\section{Introduction}

	The gradient-of-damage model motivated by Fr\'emond and Nedjar in \cite{FN96} describes the damage progression
	by microscopic motions in solid structures resulting from the growth of microcracks and
	microvoids.
	In this approach, an internal variable $z$ models the degree of damage in every material point. It is bounded in the unit interval $[0,1]$ with the following interpretation:
			the value $1$ stands for no damage,
			a value between $0$ and $1$ qualifies partial damage and
			the value $0$ indicates maximal damage.
	Beyond that, elastic deformations are described by a vector-valued function $u$ which specifies the displacement from a prescribed
	reference configuration $\Omega$.
	The evolution law for $u$ and $z$ consists of two equations:
	a hyperbolic equation for the mechanical forces and a parabolic equation for the damage process involving two subgradients.
	The time evolution of $(u,z)$ under constant temperature can be deduced from the laws of thermodynamics and is summarized
	in the following hyperbolic-parabolic PDE system (see \cite{FN96, Fr12}):
	\begin{subequations}
	\label{eqn:PDE}
	\begin{align}
	\label{eqn:pde1}
		\qquad u_{tt}&-\di\bl W_{,e}(\e(u),z)\br=\ell,\\
	\label{eqn:pde2}
		\qquad z_t&-\Delta_p z+W_{,z}(\e(u),z)+f'(z)+\xi+\varphi=0,
	\end{align}
	\end{subequations}
	where the subgradients $\xi$ and $\varphi$ are given by
	\begin{subequations}
	\label{eqn:PDEb}
	\begin{align}
	\label{eqn:pde3}
		\hspace*{4.3em}\xi&\in\partial I_{[0,\infty)}(z)&\qquad\qquad\;\;&\text{with the subdifferential}\\
		&&&\partial I_{[0,\infty)}(z)=
		\begin{cases}
			\{0\}&\text{if }z>0,\\
			(-\infty,0]&\text{if }z=0,\\
			\emptyset&\text{if }z<0,
		\end{cases}\notag\\
	\label{eqn:pde4}
		\varphi&\in\partial I_{(-\infty,0]}(z_t)&&
		\text{with the subdifferential}\\
		&&&\partial I_{(-\infty,0]}(z_t)=
		\begin{cases}
			\{0\}&\text{if }z_t<0,\\
			[0,\infty)&\text{if }z_t=0,\\
			\emptyset&\text{if }z_t>0.
		\end{cases}\hspace*{10.5em}
		\notag
	\end{align}
	\end{subequations}
	The system is supplemented with the following initial-boundary conditions:
	\begin{subequations}
	\label{eqn:PDEIBC}
	\begin{align}
	\label{eqn:PDEIBC1}
		\hspace*{1.5em}&u=b&&\text{ on }{\Gamma_\mathrm{D}}\times(0,T),\hspace*{15.7em}\\
	\label{eqn:PDEIBC2}
		&W_{,e}(\e(u),z)\cdot\nu=0&&\text{ on }\Gamma_\mathrm{N}\times(0,T),\\
	\label{eqn:PDEIBC3}
		&\nabla z\cdot\nu=0&&\text{ on }\partial\Omega,\\
		&u(0)=u^0&&\text{ in }\Omega,\\
		&u_t(0)=v^0&&\text{ in }\Omega,\\
		&z(0)=z^0&&\text{ in }\Omega.
	\end{align}
	\end{subequations}
	The hyperbolic equation \eqref{eqn:pde1} is the balance equation of forces containing inertial effects modeled by $u_{tt}$,
	the parabolic equation \eqref{eqn:pde2} describes the evolution law for the damage processes
	and \eqref{eqn:pde3} as well as \eqref{eqn:pde4} are subgradients corresponding to the constraints that the
	damage is non-negative ($z\geq 0$) and irreversible ($z_t\leq 0$).
	The symbol $\Delta_p z:=\di(|\nabla|^{p-2}\nabla z)$ denotes the $p$-Laplacian of $z$.
	The inclusions \eqref{eqn:pde3}-\eqref{eqn:pde4} are explained in more detail below.
	
	Moreover, $\ell$ denotes the exterior volume forces, $f$ a given damage-dependent potential, $\e(u)$
	describes the linearized strain tensor, i.e. $\e(u)=\frac 12(\nabla u+(\nabla u)^\mathrm{T})$,
	and ${\Gamma_\mathrm{D}}$ and ${\Gamma_\mathrm{N}}$ indicate the Dirichlet part and the Neumann part of the boundary $\partial\Omega$.
	The elastic energy density $W$ is assumed to be of the form
	\begin{align}
	\label{eqn:defW}
		W(e,z)=\frac 12 h(z)\mathbf Ce:e,
	\end{align}
	where $\mathbf C$ is the stiffness tensor and $h$ models the influence of the damage.
	We assume $h'\geq 0$ and that complete damage does not occur, i.e., $h$ is bounded from below by a positive constant.
	

	Let us give an interpretation for the system \eqref{eqn:pde2}, \eqref{eqn:pde3} and \eqref{eqn:pde4} modeling damage processes
	via gradient flows
	(for a physical motivation by means of microscopic force balance laws and constitutive relations we refer to \cite{FN96}):
	For this purpose, we observe that the inclusion \eqref{eqn:pde4} is equivalent to the complementarity problem
	\begin{align}
		&z_t\leq 0,\qquad
		\varphi z_t=0,\qquad
		\varphi\geq 0.
	\label{eqn:complementarity}
	\end{align}
	Now, we introduce the free energy $\C F$ as
	\begin{align}
		\C F(u,z):=\int_{\Omega}\bl\frac 1p|\nabla z|^p+W(\e(u),z)+f(z)+I_{[0,\infty)}(z)\br\dx,
	\label{eqn:freeEnergy}
	\end{align}
	where the indicator function $I_{[0,\infty)}$ can be interpreted as an obstacle potential, i.e., the damage variable $z$ is forced to be non-negative.
	The gradient-of-damage term $\frac 1p|\nabla z|^p$ models
	influence of the damage on its surrounding and also has a regularizing effect
	from the mathematical point of view.
	A calculation reveals $\varphi=-z_t-\zeta$ with $\zeta\in\partial_z\C F(u,z)$, where $\varphi$ is given by \eqref{eqn:pde4} and
	$\zeta$ is a subgradient of the generalized subdifferential $\partial_z\C F(u,z)$,
	i.e.
	\begin{align}
		\zeta=-\Delta_p z+W_{,z}(\e(u),z)+f'(z)+\xi,\quad\xi\in \partial I_{[0,\infty)}(z).
	\label{eqn:drivingForce}
	\end{align}
	The complementarity formulation implies
	\begin{align*}
		\begin{cases}
			\text{if } z_t<0&\text{then }z_t\in-\partial_z\C F(u,z),\\
			\text{if } z_t=0&\text{then }\zeta\leq 0,\\
			\hspace*{0.93em}z_t>0&\text{not allowed}.
		\end{cases}
	\end{align*}
	We gain the following interpretation of a non-smooth evolution:
	As long as the driving force $-\zeta$ given in \eqref{eqn:drivingForce} is non-positive, the evolution is described
	by the gradient flow $z_t=-\zeta$ tending to minimize the free energy.
	Whenever $\zeta$ becomes negative, $z_t$ is $0$.
	We remark that also an activation threshold for the damage process can be incorporated by adding a linear term to the potential $f$.
	

	The aim of this paper is to give a notion of solution to the system \eqref{eqn:PDE}-\eqref{eqn:PDEIBC}
	considered on bounded Lipschitz domains $\Omega$ in a weak sense and
	to prove existence of weak solutions.
	In the following, we summarize the main difficulties we were faced with and how they are solved:
	\begin{itemize}
		\item[$\bullet$]
			Let us point out that even on smooth domains $\Omega$ severe difficulties arise in establishing strong solutions to
			\eqref{eqn:PDE}-\eqref{eqn:PDEIBC}. Because
			due to a missing viscosity term $-\mathrm{div}(\mathbf V(z)\e(u_t))$ in \eqref{eqn:pde1}, no $L^2(H^2)$-regularity estimates for $u_t$ are available.
			This, in turn, leads to severe difficulties in obtaining higher regularity estimates for the damage variable $z$.
			In fact, to obtain $L^\infty(L^2)$-a priori estimates for the $p$-Laplacian $\Delta_p z$ and for the subgradients \eqref{eqn:pde3}-\eqref{eqn:pde4},
			we may consider Moreau-Yosida type regularizations $(\xi_\delta,\varphi_\delta)$ for $(\xi,\varphi)$.
			Then, we may test \eqref{eqn:pde2} with $(-\Delta_p z+\xi_\delta(z))_t$, integrate and use the estimates
			$\langle\varphi_\delta(z_t),-\Delta_p z_t\rangle\geq 0$ and $\langle\varphi_\delta(z_t),(\xi_\delta(z))_t\rangle\geq 0$
			and $\langle-\Delta_p z,\xi_\delta(z)\rangle\geq 0$.
			But to conclude the desired estimates via integration by parts, we would require $L^4$-bounds for the strain rate $\e(u_t)$
			(cf. \cite[Seventh a priori estimate]{RR12}).
			
			In our case the situation is even worse since no global $H^2$-elliptic regularity results are available for bounded Lipschitz domains
			and mixed boundary conditions.

			In consequence, we have to devise a concept of weak solutions.
			In the paper \cite{WIAS1520} (where the inertia term $u_{tt}$ is neglected), a notion was introduced, which allows us to formulate the double inclusion
			in system \eqref{eqn:pde2}, \eqref{eqn:pde3} and \eqref{eqn:pde4} with fairly weaker regularity assumptions.
			The notion combines a variational approach with a total energy-dissipation inequality.
			Here, we adapt this formulation to the present situation since the inertia term $u_{tt}$ in \eqref{eqn:pde1} leads to new terms in the energy-dissipation inequality.
		\item[$\bullet$]
			To prove existence of weak solutions in the sense mentioned above, we are required to establish a total energy-dissipation inequality.
			However, when using a time-discretization scheme,
			the discretized system may exhibit error terms which converge to $0$ in the time-continuous limit if
			$L^4$-bounds for the strain tensor $\e(u)$ are available.
			
			To this end, we firstly study a regularized version of system \eqref{eqn:PDE}-\eqref{eqn:PDEIBC} where an additional fourth-order
			term gives naturally rise to spatial $H^2$-regularity for the displacement variable $u$ in a weak setting.
			Finally, we perform a limit analysis to obtain a weak solution of the system \eqref{eqn:PDE}-\eqref{eqn:PDEIBC}.
		\item[$\bullet$]
			In order to perform the limit passage of the subgradients occurring in the two approximations (namely the time-discretization and the $H^2$-regularization),
			approximation techniques from \cite{WIAS1520} (see Lemma \ref{lemma:approximation} and Lemma \ref{lemma:varProp})
 			are utilized.
 			Although the precise structure of the subgradient $\xi$ can be deduced in the $H^2$-regularized version of \eqref{eqn:PDE}-\eqref{eqn:PDEIBC}
 			(see \eqref{eqn:xiDef}),
 			the situation is more involved when the passage from the $H^2$-regularized to the limit system is investigated.
 			In contrast to \cite{WIAS1520} and by the best knowledge of the authors, no strong $L^2(H^1)$-convergence of $u$
 			can be established in this limit passage due to the additional inertia term.
 			
 			We solved this problem by exploiting the precise structure of $\xi$ in the regularized system
 			such that it cancels out with other terms in \eqref{eqn:pde2} on certain parts of the domain (see \eqref{eqn:cancelTrick}).
 			For the remaining terms we apply lower-semicontinuity arguments and uniform convergence of the damage variable to pass to the limit system.
 			Only then we are able to recover the subgradient $\xi$.
	\end{itemize}
	We would like to conclude this introduction by comparing our contribution with existing mathematical works involving gradient-of-damage models
	and some of their regularization strategies in the literature
	(the choice of literature is, of course, only an excerpt and not complete):
	\begin{itemize}
		\item[$\bullet$]
			Among the pioneering works, \cite{FKS99} (see also \cite{FKNS98} for quasi-static evolutions) analyzed the damage model in \cite{FN96} in the case of an one-dimensional rod.
			The authors considered an inertia term in the force balance equation and a viscosity term in the stress tensor.
			They were able to prove local-in-time existence of solutions.
			In the case of reversible damage processes, uniqueness of solutions was proven.
			Several different approximation strategies came into play such as penalization methods for the subgradients, a regularization of the damage rate function,
			truncation and time-redarding methods.
		\item[$\bullet$]
			A 3D elasticity-damage model with inertial effects was analytically studied with a regularization of the damage rate function in \cite{BS04} and with a viscosity term in the
			stress tensor in \cite{BSS05}.
			Local-in-time existence was proven by means of Schauder fixed-point theorem and uniqueness was obtained in the case that either
			irreversibility or boundedness of the damage variable is dropped.
			Homogeneous Dirichlet boundary conditions for the displacements and smooth domains were assumed.
		\item[$\bullet$]
			Rate-independent damage models with a quasi-static balance of forces
			were investigated in \cite{Mielke06}.
			This work covers existence results for the rigid body impact problem.
			By means of a so-called energetic formulation and the usage of $\Gamma$-convergence,
			the case of complete material desintegration (complete damage) is also discussed.
		\item[$\bullet$]
			The work \cite{MT10} analyzed further rate-independent damage models and proved
			existence of energetic solutions by a new method for the construction of suitable joint recovery sequences.
			Also temporal regularity properties of the solutions were shown.
		\item[$\bullet$]
			The transition between rate-dependent and rate-independent damage models, the so-called vanishing viscosity limit, was investigated in
			\cite{KRZ11}.
			The authors assumed a quasi-static equilibrium of forces and a higher order Laplacian in the damage law (for 3D).
			Existence of energetic solutions was shown, where the proof is based on a regularization of the damage rate function.
			Furthermore, a priori estimates were derived and uniqueness results were deduced for special cases.
		\item[$\bullet$]
			A coupled system of damage, viscoelasticity with inertia and heat conduction was studied in \cite{RR12}.
			In case of isothermal processes, smooth domains and homogeneous Dirichlet boundary conditions for the displacements, existence results were established
			and uniqueness was proven for reversible damage processes.
		\item[$\bullet$]
			\cite{WIAS1520} as well as \cite{WIAS1569} explored PDE systems coupling damage processes with quasi-static elastic systems and Cahn-Hilliard equations for phase separation.
			A notion of weak solutions involving variational inequalities and a total energy-dissipation inequality was introduced.
			The authors prove existence of weak solutions and higher integrability of the strain tensors.
			They also provide abstract approximation techniques (see Lemma \ref{lemma:approximation} and Lemma \ref{lemma:varProp})
			which will be utilized in the present work to study \eqref{eqn:PDE}-\eqref{eqn:PDEIBC} involving an inertia term in the force balance equation.
			In particular, this system allows for the presence of elastic waves interacting with the damage evolution.
			
			Let us also mention that the $p$-Laplacian $-\Delta_p z:=-\mathrm{div}(|\nabla z|^{p-2}\nabla z)$
			in \eqref{eqn:pde2} describes the diffusion of damaged material parts across their proximities.
			For mathematical reasons $p$ will be chosen to be larger than the space dimension because the embedding $W^{1,p}(\Omega)\hookrightarrow C(\ol{\Omega})$
			plays an essential role in the mentioned approximation lemmas.
			By replacing the operator $-\Delta_p$ by $-(\delta\Delta_p +\Delta)$ in the damage law and performing a limit analysis $\delta\searrow 0$,
			we were even able to handle damage models with the standard Laplacian (i.e. $p=2$ as done in \cite{WIAS1569}).
			However, to keep this presentation short, we will not follow this approach here.
			
	\end{itemize}
	
	\textbf{Structure of the paper}
	
	In Section \ref{section:notation}, we introduce some notation and preliminary mathematical results from \cite{WIAS1520, WIAS1569}.
	The main part is Section \ref{section:main}. We state and justify a notion of weak solutions in Subsection \ref{section:weakNotion}.
	The proof of the existence theorem ranges from Subsection \ref{section:H2reg} to Subsection \ref{section:limit}.
	At first, we prove existence of weak solutions for an $H^2$-regularized problem by using a time-discretization scheme
	and by applying variational techniques from Section \ref{section:notation} to pass to the time-continuous system.
	Finally, we get rid of the regularization by a further limit passage which is performed in Subsection \ref{section:limit}.

\section{Notation and preliminaries}
\label{section:notation}
	Throughout this work, let $p\in(n,\infty)$ be a constant and $p'=p/(p-1)$ its dual and let $\Omega\subseteq\R^n$ ($n=1,2,3$) be a bounded Lipschitz domain.
	For the Dirichlet boundary $\Gamma_\mathrm{D}$ and the Neumann boundary $\Gamma_\mathrm{N}$ of $\partial\Omega$,
	we adopt the assumptions from \cite{Ber11}, i.e., $\Gamma_\mathrm{D}$ and $\Gamma_\mathrm{N}$ are non-empty and relatively open sets in
	$\partial\Omega$ with finitely many path-connected components such that $\Gamma_\mathrm{D}\cap \Gamma_\mathrm{N}=\emptyset$
	and $\ol{\Gamma_\mathrm{D}}\cup \ol{\Gamma_\mathrm{N}}=\partial\Omega$.
	
	The considered time interval is denoted
	by $[0,T]$ and  $\Omega_t:=\Omega \times [0,t]$ for $t \in [0,T]$.
	Form now on the time-derivative in front of the main variables and the data is denoted by $\partial_t$ and 
	the partial derivative of the density function $W$ (see \eqref{eqn:defW}) with respect to one of its variables $e$ and $z$ is
	abbreviated by $W_{,e}$ and $W_{,z}$, respectively.
	Furthermore, we define for $k\ge1$ the spaces
	\begin{align*}
		W_+^{k,p}(\Omega)&:=\big\{u\in W^{k,p}(\Omega)\,|\,u\geq 0\text{ a.e. in }\Omega\big\},\\
		W_-^{k,p}(\Omega)&:=\big\{u\in W^{k,p}(\Omega)\,|\,u\leq 0\text{ a.e. in }\Omega\big\},\\
		H_{\Gamma_\mathrm{D}}^{k}(\Omega)&:=\big\{u\in H^{k}(\Omega)\,|\,u= 0\text{ on }{\Gamma_\mathrm{D}}\text{ in the sense of traces}\big\}.
	\end{align*}

	
	The following variational and approximation results are crucial to establish existence of weak solutions.
	They are aimed to pass to the limit in certain integral inequalities with weakly convergent functions constraining the set of test-functions
	(see ``Step 1'' and ``Step 2'' in the proof of Proposition \ref{proposition:regSystem}).
	Roughly speaking, Lemma \ref{lemma:approximation} provides an approximation result in a strong topology for test-functions satisfying weakly convergent constraints.
	This enables us to pass to the limit in the corresponding integral inequalities.
	Then, Lemma \ref{lemma:varProp} provides a method to drop the limit constraints on the set of test-functions,
	where, in turn, a new integral term arises on the right-hand side.
	This method will be used to perform the limit passage of a time-discretized version of \eqref{eqn:pde2} and to recover the
	subgradients \eqref{eqn:pde3} and \eqref{eqn:pde4}.
	
	We will frequently make use of the compact embedding $$W^{1,p}(\Omega)\hookrightarrow C(\ol\Omega)$$
	without mentioning in our considerations.
	In the following,
	the notation $\{\zeta=0\}\supseteq\{f=0\}$ for functions in $L^\infty(0,T;W^{1,p}(\Omega))$ should be read as
	\begin{align}
		\big\{x\in\ol\Omega\,|\,\zeta(x,t)=0\big\}\supseteq
			\big\{x\in\ol\Omega\,|\,f(x,t)=0\big\}\qquad\textit{ for a.e. }t\in(0,T).
	\label{eqn:setInclusion}
	\end{align}
	Beyond that, the subscript $\tau$ always refers to a sequence $\{\tau_k\}_{k\in\N}$, with $\tau_k\searrow 0$ as $k\nearrow\infty$.
	
	\begin{lemma}[See {[HK13]}]
	\label{lemma:approximation}
		Let
		\begin{itemize}
			\item[$\bullet$]
				$f_\tau,f\in L^\infty(0,T;W^{1,p}_+(\Omega))$, $\tau>0$\\
				with $f_\tau(t)\to f(t)$ weakly in $W^{1,p}(\Omega)$ as $\tau\searrow 0$
				for a.e. $t\in(0,T)$,
			\item[$\bullet$]
				$\zeta \in L^\infty(0,T;W^{1,p}_+(\Omega))$
				with $\{\zeta=0\}\supseteq\{f=0\}$.
		\end{itemize}
		Then, there exist a sequence $\zeta_\tau\in L^\infty(0,T;W^{1,p}_+(\Omega))$
		and constants $\nu_{\tau,t}>0$ such that
		\begin{itemize}
			\item[$\bullet$]
				$\zeta_\tau\to \zeta$ strongly in $L^q(0,T;W^{1,p}(\Omega))$ as $\tau\searrow0$ for all $q\in[1,\infty)$,
			\item[$\bullet$]
				$\zeta_\tau\to \zeta$ weakly-star in $L^\infty(0,T;W^{1,p}(\Omega))$ as $\tau\searrow0$,
			\item[$\bullet$]
				$\zeta_\tau\leq \zeta$ a.e. in $\Omega_T$ for all $\tau>0$
				(in particular $\{\zeta_\tau=0\}\supseteq\{\zeta=0\}$),
			\item[$\bullet$]
				$\nu_{\tau,t}\zeta_\tau(t)\leq f_\tau(t)$ in $\Omega$ for a.e. $t\in(0,T)$ and for all $\tau>0$.
		\end{itemize}
		If, in addition, $\zeta\leq f$ a.e. in $\Omega_T$ then the last condition can be refined to
		$$
			\zeta_\tau\leq f_\tau\text{ a.e. in }\Omega_T\text{ for all }\tau>0.
		$$
	\end{lemma}

	\begin{lemma}[See {[HK13]}]
	\label{lemma:varProp}
		Let $f\in L^{p'}(\Omega;\mathbb R^n)$, $g\in L^1(\Omega)$ and $z\in W_+^{1,p}(\Omega)$
		with $f\cdot\nabla z\geq 0$ a.e. in $\Omega$ and $\{f=0\}\supseteq\{z=0\}$ in an a.e. sense. Furthermore, we assume that
		\begin{align*}
			\int_\Omega \big(f\cdot\nabla\zeta+g\zeta\big)\,\mathrm dx\geq 0\quad
				\text{for all }\zeta\in W_-^{1,p}(\Omega)\text{ with }\{\zeta=0\}\supseteq \{z=0\}.
		\end{align*}
		Then
		\begin{align*}
			\int_\Omega \big(f\cdot\nabla\zeta+g\zeta\big)\,\mathrm dx
				\geq\int_{\{z=0\}}\mathrm{max}\{0,g\} \zeta\,\mathrm dx\quad
				\text{for all }\zeta\in W_-^{1,p}(\Omega).
		\end{align*}
	\end{lemma}
	\begin{remark}
		In [HK13], $g$ is assumed to be in $L^p(\Omega)$. But the proof extends to $g\in L^1(\Omega)$ without any modifications.
	\end{remark}
	
\section{Analysis of the hyperbolic-parabolic system}
\label{section:main}
\subsection{Notion of weak solutions and existence theorem}
\label{section:weakNotion}
	
	\textbf{Assumptions on the coefficients}\\
	Besides the general setting introduced at the beginning of the previous section,
	we now state the conditions needed for the coefficient functions in the PDE system \eqref{eqn:pde1}-\eqref{eqn:pde2}
	in order to develop a weak notion and to prove existence results:
	
	We assume $f\in C^1([0,1],\R_+)$ for the damage potential function (occurring in \eqref{eqn:pde2}) and $W$ to be given by \eqref{eqn:defW}
	with $h\in C^1([0,1];\R)$ and $h\geq \eta$ on $[0,1]$
	for a constant $\eta>0$.
	Moreover, we will use the assumption $h'\geq 0$ on $[0,1]$.
	The stiffness tensor $\CC$ should satisfy the usual symmetry and coercivity assumptions, i.e.
	$\CC_{ijlk}=\CC_{jilk}=\CC_{lkij}$ and $e:\CC e\geq c_0|e|^2$ for all $e\in \R_\mathrm{sym}^{d\times d}$
	and constant $c_0>0$.
	
	\textbf{Weak formulation}\\
	The main idea for a weak formulation is to rewrite the doubly nonlinear
	differential inclusion in system \eqref{eqn:pde2}-\eqref{eqn:pde4} into a variational inequality and a total energy inequality.
	This kind of notion was introduced in \cite{WIAS1520} and is adapted to the present situation.

	\begin{lemma}
	\label{lemma:notions}
		Let the data
		\begin{align*}
			&u^0\in H^1(\Omega;\R^n),\;v^0\in L^2(\Omega;\R^n),\;z^0\in W^{1,p}(\Omega)
			\text{ with }0\leq z^0\leq 1\text{ a.e. in }\Omega,\\
			&\ell\in L^2(0,T;L^2(\Omega;\R^n)),\\
			&b\in H^1(0,T;H^1(\Omega;\R^n))\cap H^2(0,T;L^2(\Omega;\R^n))
		\end{align*}
		and
		the functions
		\begin{align*}
			&u\in L^2(0,T;H^2(\Omega;\Rn))\cap H^1(0,T;H^1(\Omega;\Rn))\cap H^2(0,T;L^2(\Omega;\Rn)),\\
			&z\in L^p(0,T;W^{2,p}(\Omega))\cap W^{1,p}(0,T;W^{1,p}(\Omega)),\\
			&\xi,\varphi\in L^2(0,T;L^2(\Omega))
		\end{align*}
		be given.
		The following statements are equivalent:
		\begin{enumerate}
			\item[(i)] The functions $(u,z,\xi,\varphi)$ satisfy the PDE system \eqref{eqn:PDE}-\eqref{eqn:PDEIBC} pointwise in an a.e. sense
				to the data $(u^0,v^0,z^0,\ell,b)$.
			\item[(ii)] The functions $(u,z,\xi,\varphi)$ satisfy the following properties:
				\begin{itemize}
					\item[$\bullet$] initial and Dirichlet boundary conditions:
						\begin{align}
						\label{eqn:initialBoundary}
							u(0)=u^0,\;\partial_t u(0)=v^0,\;z(0)=z^0,\;u=b\text{ on }\GammaD\times(0,T),
						\end{align}
					\item[$\bullet$] force balance:
						\begin{align}
						\label{eqn:weakMomentumBalance}
							\langle \partial_{tt}u,\zeta\rangle_{H_\GammaD^1}+\int_\Omega W_{,e}(\e(u),z):\e(\zeta)\dx=\int_\Omega \ell\cdot\zeta\dx
						\end{align}
						for all $\zeta\in H_{\Gamma_\mathrm{D}}^1(\Omega;\R^n)$ and for a.e. $t\in(0,T)$,
					\item[$\bullet$] damage evolution law:
						\begin{align}
						\label{eqn:weakDamageLaw}
							0=\int_\Omega\Big((\partial_t z)\zeta+|\nabla z|^{p-2}\nabla z\cdot\nabla\zeta+W_{,z}(\e(u),z)\zeta+f'(z)\zeta\Big)\dx
							+\langle\xi,\zeta\rangle_{W^{1,p}}+\langle\varphi,\zeta\rangle_{W^{1,p}}
						\end{align}
						for all $\zeta\in W^{1,p}(\Omega)$ and for a.e. $t\in(0,T)$,
					\item[$\bullet$] conditions for the subgradients and the damage variable:
						\begin{subequations}
						\label{eqn:weakVarIneq1}
						\begin{align}
						\label{eqn:weakVarIneq1a}
							\;\qquad\langle\varphi(t),\zeta\rangle_{W^{1,p}}&\leq 0
							&&\text{for all }\zeta\in W_-^{1,p}(\Omega),\;\text{a.e. }t\in(0,T),\;\\
						\label{eqn:weakVarIneq1b}
							\partial_t z&\leq 0
							&&\text{a.e. in }\Omega\times(0,T)
						\end{align}
						\end{subequations}
						and
						\begin{subequations}
						\label{eqn:weakVarIneq2}
						\begin{align}
							\qquad\langle\xi(t),\zeta-z(t)\rangle_{W^{1,p}}&\leq 0\quad\quad\;
						\label{eqn:weakVarIneq2a}
							&&\text{for all }\zeta\in W_+^{1,p}(\Omega),\;\text{a.e. }t\in(0,T),\;\;\qquad\\
						\label{eqn:weakVarIneq2b}
							z&\geq 0
							&&\text{a.e. in }\Omega\times(0,T),
						\end{align}
						\end{subequations}
					\item[$\bullet$]
						total energy-dissipation balance:
						\begin{align}
							\C F(t)+\C K(t)+\C D(0,t)= 
								\C F(0)+\C K(0)+\C W_{ext}(0,t)
						\label{eqn:weakEnergyBalance}
						\end{align}
						for a.e. $t\in(0,T)$ with
						\begin{align*}
							&\text{free energy:}
								&&\C F(t):=\int_\Omega\bl\frac 1p|\nabla z(t)|^p+W(\e(u(t)),z(t))+f(z(t))\br\dx,\\
							&\text{kinetic energy:}
								&&\C K(t):=\int_\Omega\frac 12|\partial_t u(t)|^2\dx,\\
							&\text{dissipated energy:}
								&&\C D(0,t):=\int_0^t\int_\Omega|\partial_t z|^2\dxs,\\
							&\text{external work:}
								&&\C W_{ext}(0,t):=\int_0^t\int_\Omega W_{,e}(\e(u),z):\e(\partial_t b)\dxs
									-\int_0^t\int_{\Omega}\partial_t u\cdot \partial_{tt}b\dxs\\
									&&&\qquad\qquad\qquad+\int_\Omega \partial_t u(t)\cdot \partial_t b(t)\dx
									-\int_\Omega \partial_t u(0)\cdot \partial_t b(0)\dx\\
									&&&\qquad\qquad\qquad+\int_0^t\int_{\Omega}\ell\cdot \partial_t (u-b)\dxs.
						\end{align*}
				\end{itemize}
		\end{enumerate}
	\end{lemma}
	\begin{remark}
	\label{remark:varIneq}
		The essential feature in the notion presented in (ii) is that we only require
		$\langle\varphi,\zeta\rangle_{W^{1,p}}\leq 0$ in \eqref{eqn:weakVarIneq1} instead of the full variational inequality
		$\langle\varphi,\zeta-\partial_t z\rangle_{W^{1,p}}\leq 0$.
		In view of the complementarity formulation \eqref{eqn:complementarity} it means that the condition $\langle\varphi,\partial_t z\rangle_{W^{1,p}}=0$ is dropped.
		However, as we will see in the proof, the total energy-dissipation balance allows us to recover this identity.
	\end{remark}

	\begin{remark}
		Let us give a physical interpretation of the term $\C W_{ext}(0,t)$.
		To this end, remember that the Cauchy stress tensor is given by
		$$
			\sigma=W_{,e}(\e(u),z).
		$$
		Now, by using integration by parts, the force balance equation \eqref{eqn:pde1} as well as
		the boundary conditions in \eqref{eqn:PDEIBC}, $\C W_{ext}(0,t)$ transforms into
			\begin{align*}
				\C W_{ext}(0,t)={}&\int_0^t\int_{\GammaD}\big(\sigma\cdot \partial_t b\big)\cdot\nu\dxs
					+\int_0^t\int_\Omega\Big(-\di(\sigma)\cdot \partial_t b+\partial_{tt }u\cdot \partial_t b\Big)\dxs\\
					&+\int_0^t\int_{\Omega}\ell\cdot \partial_t (u-b)\dxs\\
					={}&\int_0^t\int_{\GammaD}\big(\sigma\cdot\nu\big)\cdot \partial_t b\dxs+\int_0^t\int_{\Omega}\ell\cdot \partial_t u\dxs.
			\end{align*}
			The first integral term may be interpreted as the work performed by the prescribed (time-dependent) Dirichlet boundary data
			and the second integral term is the work performed by the external forces $\ell$.
	\end{remark}
	\textbf{Proof of Lemma \ref{lemma:notions}}\\
		\underline{To ``(ii) implies (i)'':}
		
		The equations \eqref{eqn:pde1} and \eqref{eqn:pde2} can be recovered by standard arguments
		from \eqref{eqn:weakMomentumBalance} and \eqref{eqn:weakDamageLaw}.
		The inclusion \eqref{eqn:pde3} follows by the variational inequality \eqref{eqn:weakVarIneq2}.
		
		It remains to prove the validity of the inclusion \eqref{eqn:pde4}.
		As mentioned in Remark \ref{remark:varIneq}, we need to show $\langle\varphi,\partial_t z\rangle_{W^{1,p}}=0$.
		
		To this end, denote the free energy functional without the indicator part (see \eqref{eqn:freeEnergy}) by
		$$
			\widetilde{\C F}(u,z):=\int_{\Omega}\bl\frac 1p|\nabla z|^p+W(\e(u),z)+f(z)\br\dx.
		$$
		The G\^{a}teaux derivatives $\mathrm d_u\widetilde{\C F}$ and $\mathrm d_z\widetilde{\C F}$ are given as follows:
		\begin{align*}
			&\langle\mathrm d_u\widetilde{\C F}(u,z),\zeta\rangle_{H^1}
				=\int_\Omega W_{,e}(\e(u),z):\e(\zeta)\dx,\\
			&\langle\mathrm d_z\widetilde{\C F}(u,z),\zeta\rangle_{W^{1,p}}
				=\int_\Omega\bl|\nabla z|^{p-2}\nabla z\cdot\nabla\zeta+W_{,z}(\e(u),z)\zeta+f'(z)\zeta\br\dx.
		\end{align*}
		Testing \eqref{eqn:weakMomentumBalance} with $\partial_t u-\partial_t b$ (note that $\partial_t u-\partial_t b=0$ on $\GammaD$) yields
		\begin{align}
			\langle\mathrm d_u\widetilde{\C F}(u,z),\partial_t u\rangle_{H^1}
				={}&\langle\mathrm d_u\widetilde{\C F}(u,z),\partial_t b\rangle_{H^1}
				+\int_\Omega \ell\cdot\partial_t (u-b)\dx
				-\int_\Omega\frac{\mathrm d}{\mathrm dt}\frac 12|\partial_t u|^2\dx\notag\\
		\label{eqn:test1}
			&+\langle \partial_{tt} u,\partial_t b\rangle_{H^1}.
		\end{align}
		Testing \eqref{eqn:weakDamageLaw} with $\partial_t z$ shows
		\begin{align}
		\label{eqn:test2}
			 \langle \mathrm d_z\widetilde{\C F}(u,z),\partial_t z\rangle_{W^{1,p}}=-\int_\Omega|\partial_t z|^2\dx
			-\langle\xi,\partial_t z\rangle_{W^{1,p}}-\langle\varphi,\partial_t z\rangle_{W^{1,p}}.
		\end{align}
		We obtain by adding \eqref{eqn:test1} and \eqref{eqn:test2}, integrating over time and using the chain rule
		for $\frac{\mathrm d}{\mathrm d t}\widetilde{\C F}(u(t),z(t))$:
		\begin{align}
			\widetilde{\C F}(u(t),z(t))-\widetilde{\C F}(u(0),z(0))
			={}&\int_0^t\Big(\langle\mathrm d_u\widetilde{\C F}(u(s),z(s)),\partial_t u(s)\rangle_{H^1}
				+\langle\mathrm d_z\widetilde{\C F}(u(s),z(s)),\partial_t z(s)\rangle_{W^{1,p}}\Big)\ds\notag\\
			={}&
				-\int_\Omega\frac 12|\partial_t u(t)|^2\dx+\int_\Omega\frac 12|\partial_t u(0)|^2\dx
				+\int_0^t\int_\Omega W_{,e}(\e(u),z):\e(\partial_t b)\dxs\notag\\
				&+\int_0^t\int_\Omega \ell\cdot\partial_t (u-b)\dxs
				+\int_0^t\int_\Omega \partial_{tt} u\cdot \partial_t b\dxs
				-\int_0^t\int_\Omega|\partial_t z|^2\dxs\notag\\
				&-\langle\xi,\partial_t z\rangle_{W^{1,p}}-\langle\varphi,\partial_t z\rangle_{W^{1,p}}.
		\label{eqn:energyDifference}
		\end{align}
		Furthermore, integration by parts in time shows
		\begin{align}
			\int_0^t\int_\Omega \partial_{tt}u\cdot \partial_t b\dxs
			=-\int_0^t\int_\Omega \partial_{t}u\cdot \partial_{tt}b\dxs
				+\int_\Omega \partial_{t}u(t)\cdot \partial_t b(t)\dx
				-\int_\Omega \partial_{t}u(0)\cdot \partial_t b(0)\dx.
		\label{eqn:uttbt}
		\end{align}
		By applying \eqref{eqn:uttbt} to \eqref{eqn:energyDifference}, we obtain
		\begin{align*}
			\C F(t)-\C F(0)
			= -\C K(t)+\C K(0)
				+\C W_{ext}(0,t)
				-\C D(0,t)
				-\langle\xi,\partial_{t}z\rangle_{W^{1,p}}-\langle\varphi,\partial_{t}z\rangle_{W^{1,p}}.
		\end{align*}
		Subtracting the total energy-dissipation balance \eqref{eqn:weakEnergyBalance} from above yields
		\begin{align}
			\langle\xi,\partial_{t}z\rangle_{W^{1,p}}+\langle\varphi,\partial_{t}z\rangle_{W^{1,p}}=0.
		\label{eqn:xiPhi}
		\end{align}
		In order to obtain $\langle\varphi,\partial_{t}z\rangle_{W^{1,p}}=0$, we will prove $\langle\xi,\partial_{t}z\rangle_{W^{1,p}}=0$ and then applying \eqref{eqn:xiPhi}.
		Since $\xi\in L^2(0,T;L^2(\Omega))$ by assumption, we find by \eqref{eqn:weakVarIneq2}
		$$
			\int_\Omega\xi(\zeta-z)\dx\leq 0
		$$
		for all $\zeta\in L_+^2(\Omega)$ and for a.e. $t\in(0,T)$.
		This shows $\xi=0$ a.e. in $\{z>0\}$ and $\xi\leq 0$ a.e. in $\{z=0\}$.
		Since $\partial_{t}z=0$ holds a.e. in $\{z=0\}$, we obtain
		$$
			\langle\xi,\partial_{t}z\rangle_{W^{1,p}}=\int_\Omega\xi \partial_{t}z\dx=0.
		$$
		\underline{To ``(i) implies (ii)'':}
		
		The force balance equation \eqref{eqn:weakMomentumBalance}, the damage evolution law \eqref{eqn:weakDamageLaw}
		and the variational inequality \eqref{eqn:weakVarIneq2} follow from \eqref{eqn:pde1},\eqref{eqn:pde2}, \eqref{eqn:pde3} and the boundary conditions
		\eqref{eqn:PDEIBC1}-\eqref{eqn:PDEIBC3} without much effort.
		It remains to show the variational property \eqref{eqn:weakVarIneq1} and the total energy-dissipation balance \eqref{eqn:weakEnergyBalance}.
		
		In fact, \eqref{eqn:pde4} implies the complementarity formulation
		\begin{align}
			\partial_{t}z\leq 0,\qquad
			\langle\varphi,\partial_{t}z\rangle_{W^{1,p}}=0,\qquad
			\langle\varphi,\zeta\rangle_{W^{1,p}}\geq 0
		\label{eqn:complementarity2}
		\end{align}
		for all $\zeta\in W_-^{1,p}(\Omega)$.
		In particular, \eqref{eqn:weakVarIneq1} is shown.
		
		To show \eqref{eqn:weakEnergyBalance}, we consider the calculation \eqref{eqn:energyDifference} which follows with the same arguments as in (i).
		Taking into account the conditions $\langle\varphi,\partial_{t}z\rangle_{W^{1,p}}=0$ (see \eqref{eqn:complementarity2}) and
		$\langle\xi,\partial_{t}z\rangle_{W^{1,p}}=0$ (follows as in (i)), we obtain the total energy-dissipation balance.
	\ep
	
	From Lemma \ref{lemma:notions} we make the following observations:
	\begin{itemize}
		\item[$\bullet$]
			The weak formulation in Lemma \ref{lemma:notions} (ii) requires much less regularity
			as assumed there. Since we avoid the full variational inequality $\langle\varphi,\zeta-\partial_{t}z\rangle_{W^{1,p}}\leq 0$
			by means of the total energy-dissipation balance, the notion makes still sense for
			$z\in L^\infty(0,T;W^{1,p}(\Omega))\cap H^1(0,T;L^2(\Omega))$.
			This is very important because the a priori estimates will give no better regularity in this weak setting
			(remember that $\Omega$ is only a bounded Lipschitz domain).
		\item[$\bullet$]
			The regularity
			$u\in L^\infty(0,T;H^1(\Omega;\R^n))\cap W^{1,\infty}(0,T;L^2(\Omega;\R^n))\cap H^2(0,T;(H_{\Gamma_\mathrm{D}}^1(\Omega;\R^n))^*)$
			for $u$ is sufficient for the notion in Lemma \ref{lemma:notions} (ii).
			We observe that the term $\C W_{ext}$ avoids the second time-derivative of $u$. In fact, we may also write
			(by using \eqref{eqn:uttbt})
			\begin{align*}
			\begin{split}
				\C W_{ext}(0,t)
				={}&\int_0^t\int_\Omega W_{,e}(\e(u),z):\e(\partial_{t}b)\dxs
					+\int_0^t\langle \partial_{tt}u(s),\partial_{t}b(s)\rangle_{H^1}\ds\\
					&+\int_0^t\int_{\Omega}\ell\cdot \partial_{t}(u-b)\dxs
			\end{split}
			\end{align*}
			provided that $u\in H^2(0,T;(H^1(\Omega;\R^n))^*)$ (note that $\partial_{t}b$ is not necessarily $0$ on $\GammaD$).
		\item[$\bullet$]
			Since $\xi$ is assumed to be in $L^2(0,T;L^2(\Omega))$, the $W^{1,p}$-dual product rewrites as
			$$
				\langle\xi(t),\zeta\rangle_{W^{1,p}}
				=\int_\Omega \xi(t)\zeta\dx.
			$$
			The regularity assumptions for $\xi$ can be weakened to $L^1(0,T;L^1(\Omega))$ since the
			integral term on the right-hand side exists in this case for all $\zeta\in W^{1,p}(\Omega)$ provided that $p\in(n,\infty)$
			by employing the embedding $W^{1,p}(\Omega)\hookrightarrow L^\infty(\Omega)$.
		\item[$\bullet$]
			Conditions \eqref{eqn:weakDamageLaw} and \eqref{eqn:weakVarIneq1a} can be reformulated as one inequality as
			\begin{align}
				0\leq\int_\Omega\Big((\partial_{t}z)\zeta+|\nabla z|^{p-2}\nabla z\cdot\nabla\zeta+W_{,z}(\e(u),z)\zeta+f'(z)\zeta\Big)\dx
				+\langle\xi,\zeta\rangle_{W^{1,p}}
				\label{eqn:weakDamageLaw2}
			\end{align}
			holding for all $\zeta\in W_-^{1,p}(\Omega)$ and for a.e. $t\in(0,T)$.
			This eliminates $\varphi$ in the weak formulation.
	\end{itemize}
	A further not so obvious observation is stated in the following lemma:
	\begin{lemma}
		Let $u$, $z$, $\xi$ and $\varphi$ as well as $u^0$, $v^0$, $z^0$, $\ell$ and $b$ be as in Lemma \ref{lemma:notions}.
		Suppose that the conditions (ii) in Lemma \ref{lemma:notions} except the total energy-dissipation balance are satisfied.
		Instead we assume for a.e. $t\in(0,T)$ the total energy-dissipation inequality
		\begin{align}
			\C F(t)+\C K(t)+\C D(0,t)\leq
				\C F(0)+\C K(0)+\C W_{ext}(0,t),
		\label{eqn:weakEnergyIneq}
		\end{align}
		where the functions $\C F(\cdot)$, $\C K(0,\cdot)$, $\C D(0,\cdot)$ and $\C W_{ext}(0,\cdot)$
		are defined in Lemma \ref{lemma:notions} (ii).
		
		Then, the total energy-dissipation balance is satisfied, i.e., \eqref{eqn:weakEnergyIneq} is an equality.
	\end{lemma}
	\begin{proof}
		We find \eqref{eqn:energyDifference} and $\langle\xi,\partial_{t}z\rangle_{W^{1,p}}=0$
		with the same argumentation as in the proof of Lemma \ref{lemma:notions} (see ``(ii) implies (i)'').
		
		Then, taking
		$$
			\langle\varphi,\partial_{t}z\rangle_{W^{1,p}}\leq0
		$$
		into account (which follows from \eqref{eqn:weakVarIneq1} tested with $\varphi=\partial_{t}z$), 
		we obtain the desired ``$\geq$''-part of the total energy-dissipation balance.
		\ep
	\end{proof}

	These observations motivate the following notion of weak solutions which is equivalent to
	the PDE system \eqref{eqn:PDE}-\eqref{eqn:PDEIBC} provided sufficient regularity.

	\begin{definition}[Notion of weak solutions]
	\label{def:weakSolution}
		Let the data $(u^0,v^0,z^0,\ell,b)$ be given as in Lemma \ref{lemma:notions}.
		A weak solution of the PDE system \eqref{eqn:PDE}-\eqref{eqn:PDEIBC} is a triple $(u,z,\xi)$ of functions
		\begin{align*}
			&u\in L^\infty(0,T;H^1(\Omega;\R^n))\cap W^{1,\infty}(0,T;L^2(\Omega;\R^n))\cap H^2(0,T;(H_{\Gamma_\mathrm{D}}^1(\Omega;\R^n))^*),\\
			&z\in L^\infty(0,T;W^{1,p}(\Omega))\cap H^1(0,T;L^2(\Omega)),\\
			&\xi\in L^1(\Omega\times(0,T))
		\end{align*}
		satisfying \eqref{eqn:initialBoundary}, \eqref{eqn:weakMomentumBalance}, \eqref{eqn:weakVarIneq1b}, \eqref{eqn:weakVarIneq2},
		\eqref{eqn:weakDamageLaw2} and the total energy-dissipation inequality
		\begin{align}
			\C F(t)+\C K(t)+\C D(0,t)\leq
				\C F(0)+\C K(0)+\C W_{ext}(0,t)
		\label{eqn:weakEnergyInequality}
		\end{align}
		for a.e. $t\in(0,T)$.
	\end{definition}
	
	The main aim of this work is to prove existence of weak solutions in the sense of Definition \ref{def:weakSolution}.
	\begin{theorem}[Existence of weak solutions]
	\label{theorem:mainResult}
		Let the assumptions in Section \ref{section:notation} and on the coefficients $f$, $h$ and $\CC$ be satisfied.
		Furthermore, let the data $(u^0,v^0,z^0,\ell,b)$ according to Lemma \ref{lemma:notions} be given.
		Then, there exists a weak solution of system \eqref{eqn:PDE}-\eqref{eqn:PDEIBC}
		in the sense of Definition \ref{def:weakSolution}.
	\end{theorem}
	We emphasize that proving the total energy-dissipation balance \eqref{eqn:weakEnergyBalance}
	in the weak setting seems to be out of reach by the authors' best knowledge. However, by utilizing non-local or regularized versions of the Laplacian operator
	in the damage law \eqref{eqn:pde2}, the energy balance can be recovered as shown in \cite{KRZ11, KRZ13}.

\subsection{Existence of weak solutions for an $\bm{H^2}$-regularized system}
\label{section:H2reg}

	We firstly study a regularized version of the PDE system \eqref{eqn:PDE}-\eqref{eqn:PDEIBC} in order to prove Theorem \ref{theorem:mainResult}.
	The enhanced regularity allows us to control an error term which occurs in the discrete version of the energy-dissipation inequality.
	The passage to the limit system is then performed in the next subsection.


	The regularized PDE system is described in a classical notion by a quadruple $(u,z,\xi,\varphi)$ of functions satisfying
	the following equations pointwise:
	\begin{subequations}
	\label{eqn:regPDE}
	\begin{flalign}
		\qquad \partial_{tt}u&-\di(W_{,e}(\e(u),z))+\delta \di(\di(\nabla(\nabla u)))=\ell,&\\
		\qquad \partial_{t}z&-\Delta_p z+W_{,z}(\e(u),z)+f'(z)+\xi+\varphi=0,&\\
		\qquad \xi&\in\partial I_{[0,\infty)}(z),&\\
		\qquad \varphi&\in\partial I_{(-\infty,0]}(\partial_{t}z)&
	\end{flalign}
	\end{subequations}
	with the initial-boundary conditions
	\begin{subequations}
	\label{eqn:regPDEIBC}
	\begin{flalign}
		&\qquad u=b&&\hspace*{-9em}\text{ on }{\Gamma_\mathrm{D}}\times(0,T),\hspace*{10em}\\
		&\qquad \big(W_{,e}(\e(u),z)-\delta \di(\nabla(\nabla u))\big)\cdot\nu=0&&\hspace*{-9em}\text{ on }\Gamma_\mathrm{N}\times(0,T),\\
		&\qquad \delta \nabla(\nabla u)\cdot\nu=0&&\hspace*{-9em}\text{ on }\partial\Omega\times(0,T),\\
		&\qquad \nabla z\cdot\nu=0&&\hspace*{-9em}\text{ on }\partial\Omega,\\
		&\qquad u(0)=u^0,\;\partial_{t}u(0)=v^0,\;z(0)=z^0&&\hspace*{-9em}\text{ in }\Omega.
	\end{flalign}
	\end{subequations}
	In order to treat the fourth order regularization term ``$\delta \di(\di(\nabla(\nabla u)))$'' with the given constant $\delta>0$ analytically, we introduce
	the linear operator $A:H^2(\Omega;\R^n)\to (H^2(\Omega;\R^n))^*$ by
	$$
		\langle Au,v\rangle_{H^2}
		:=\int_\Omega\langle\nabla(\nabla u),\nabla(\nabla v)\rangle_{\R^{n\times n\times n}}\dx:=
			\sum_{1\leq i,j,k\leq n}\int_\Omega \partial_{x_ix_j} u_k\partial_{x_ix_j} v_k\dx.
	$$
	According to the regularization we modify the weak formulation in Definition \ref{def:weakSolution} as follows:
	\begin{definition}[Notion of weak solutions for the regularized system]
	\label{def:regWeakSolution}
		We consider the given data $(u^0,v^0,z^0,\ell,b)$ as in Lemma \ref{lemma:notions}.
		Additionally, we assume $u^0\in H^2(\Omega;\Rn)$.
		A weak solution of the regularized PDE system \eqref{eqn:regPDE}-\eqref{eqn:regPDEIBC} is a triple $(u,z,\xi)$ of functions
		\begin{align*}
			&u\in L^\infty(0,T;H^2(\Omega;\R^n))\cap W^{1,\infty}(0,T;L^2(\Omega;\R^n))\cap H^2(0,T;(H_{\Gamma_\mathrm{D}}^2(\Omega;\R^n))^*),\\
			&z\in L^\infty(0,T;W^{1,p}(\Omega))\cap H^1(0,T;L^2(\Omega)),\\
			&\xi\in L^1(\Omega\times(0,T))
		\end{align*}
		satisfying \eqref{eqn:initialBoundary}, \eqref{eqn:weakVarIneq1b}, \eqref{eqn:weakVarIneq2}, \eqref{eqn:weakDamageLaw2},
		the regularized version of the forces balance equation
		\begin{align}
		\label{eqn:weakRegMomentumBalance}
			\langle \partial_{tt} u,\zeta\rangle_{H_\GammaD^2}+\int_\Omega W_{,e}(\e(u),z):\e(\zeta)\dx+\delta \langle Au,\zeta\rangle_{H_\GammaD^2}=\int_\Omega \ell\cdot\zeta\dx
		\end{align}
		for all $\zeta\in H_{\Gamma_\mathrm{D}}^2(\Omega;\R^n)$ and a.e. $t\in(0,T)$
		and the regularized version of the total energy-dissipation inequality
		\begin{align} 
			\C F_\delta(t)+\C K(t)+\C D(0,t)\leq
				\C F_\delta(0)+\C K(0)+\C W_{ext}^\delta(0,t)
		\label{eqn:weakRegEnergyInequality}
		\end{align}
		for a.e. $t\in(0,T)$, where $\C F$, $\C K$, $\C D$ and $\C W_{ext}$ are given as in Definition \ref{lemma:notions} and
		\begin{align*}
			&\C F_\delta(t):=\C F(t)+\frac\delta2\langle A u(t),u(t)\rangle_{H^2},\\
			&\C W_{ext}^\delta(0,t):=\C W_{ext}(0,t)+\delta\int_0^t \langle Au(t),\partial_t b(t)\rangle_{H^2}\dt.
		\end{align*}
	\end{definition}
	
	The goal of this section is to prove the following proposition:
	\begin{proposition}[Existence of weak solutions for the regularized system]
	\label{proposition:regSystem}
		Let the assumptions in Section \ref{section:notation} and at the beginning of Section \ref{section:weakNotion} on the coefficients $f$, $h$ and $\CC$ be satisfied.
		Furthermore, let $\delta>0$ and the data $(u^0,v^0,z^0,\ell,b)$ according to Lemma \ref{lemma:notions} be given.
		Additionally, assume $u^0\in H^2(\Omega;\Rn)$, $\ell\in C^{0,1}(0,T;L^2(\Omega;\Rn))$ and $b\in C^{2,1}(0,T;H^2(\Omega;\Rn))$.
		Then, there exists a weak solution of system \eqref{eqn:regPDE}-\eqref{eqn:regPDEIBC}
		in the sense of Definition \ref{def:regWeakSolution}.
		
		Moreover, for this weak solution, the subgradient $\xi$ has the form
		\begin{align} 
			\xi=-\chi_{\{z=0\}}\mathrm{max}\Big\{0,W_{,z}(\varepsilon(u),z)+f'(z)\Big\}.
		\label{eqn:xiDef}
		\end{align}
	\end{proposition}

	We will complete the proof of Proposition \ref{proposition:regSystem} at the end of this section.
	The proof is based on a time-discretization scheme.
	To this end,
	let $\{0,\tau,2\tau,\ldots,T\}$ be an equidistant partition with fineness $\tau:=T/M>0$ of the interval $[0,T]$.
	The final time index is denoted by $M\in\N$.
	
	By a recursive minimization procedure starting from the initial values $(u_{\tau,\delta}^0,z_{\tau,\delta}^0):=(u^0,z^0)$ and $u_{\tau,\delta}^{-1}:=u^0-\tau v^0$, we obtain
	functions $(u_{\tau,\delta}^m,z_{\tau,\delta}^m)$ for every $m=0,\ldots,M$.
	To this end, we fix an $m\in\{1,\ldots, M\}$ and define the functional $\C F_{\tau,\delta}^m:H^2(\Omega;\R^n)\times W^{1,p}(\Omega)\to\R$ by
	\begin{align*}
		\C F_{\tau,\delta}^m(u,z):={}&\int_\Omega\bl\frac 1p|\nabla z|^p+W(\e(u),z)+f(z)-\ell_\tau^m\cdot u\br\dx+\frac\delta2\langle Au,u\rangle_{H^2}\\
		&+\frac\tau2\left\|\frac{z-z_{\tau,\delta}^{m-1}}{\tau}\right\|_{L^2}^2+\frac{\tau^2}{2}\left\|\frac{u-2u_{\tau,\delta}^{m-1}+u_{\tau,\delta}^{m-2}}{\tau^2}\right\|_{L^2}^2.
	\end{align*}
	A minimizer of $\C F_{\tau,\delta}^m$ in the subspace
	$$
		\C U_{\tau}^m\times\C Z_{\tau,\delta}^m:=
			\Big\{u\in \Hb\;|\;u|_{\Gamma_\mathrm{D}}=b_{\tau}^m)|_{\Gamma_\mathrm{D}}\Big\}\times \Big\{z\in \Wa\;|\;0\leq z\leq z_{\tau,\delta}^{m-1}\Big\}
	$$
	obtained by the direct method is denoted by $(u_{\tau,\delta}^m,z_{\tau,\delta}^m)$.
	The velocity field $v_{\tau,\delta}^m$ is defined as the time-discrete derivative $(u_{\tau,\delta}^m-u_{\tau,\delta}^{m-1})/\tau$
	and the discretizations $b_{\tau}^m$ and $\ell_{\tau}^m$ are set to
	\begin{align}
		b_{\tau}^m:=b(m\tau),
		\qquad
		\ell_{\tau}^m:=\ell(m\tau).
	\end{align}
	
	For a discretization $w^m\in\{u_{\tau,\delta}^m,v_{\tau,\delta}^m,z_{\tau,\delta}^m,\ell_{\tau}^m,b_{\tau}^m\}$, we introduce the piecewise constant
	interpolations $w$, $w^-$
	and the linear interpolation $\widehat w$ with respect to the time variable as
	\begin{align*}
		w(t)&:=w^m&&\text{ with }m=\left\lceil t/\tau\right\rceil,\\
		w^-(t)&:=w^{\max\{0,m-1\}}&&\text{ with }m=\left\lceil t/\tau\right\rceil,\\
		\widehat w(t)&:=\beta w^m+(1-\beta)w^{\max\{0,m-1\}}&&\text{ with }
			m=\left\lceil t/\tau\right\rceil,\;\beta=\frac{t-(m-1)\tau}{\tau}
	\end{align*}
	and the piecewise constant functions $t_\tau$ and $t_\tau^-$ as
	\begin{align*}
		t_\tau&:=\left\lceil t/\tau\right\rceil\tau=\min\{m\tau\,|\,m\in\mathbb N_0\text{ and }m\tau\geq t\},\\
		t_\tau^-&:=\max\{0,t_\tau-\tau\}.
	\end{align*}
	We would like to remark that by above definition we have $w(t)=w(t_\tau)$ for all $t\in[0,T]$.

	Within this notation, the velocity field satisfies
	\begin{align*}
		\partial_t \widehat v_\tau(t)=\frac{u_\tau^m-2u_\tau^{m-1}+u_\tau^{m-2}}{\tau^2}
	\end{align*}
	with $m=\left\lceil t/\tau\right\rceil$ and $t\in(0,T]$.
	
	Since $\delta>0$ is assumed to be constant in this section, we mostly omit the subscript $\delta$ in $u_{\tau,\delta}^m$, $v_{\tau,\delta}^m$, $z_{\tau,\delta}^m$ and
	$\C F_{\tau,\delta}^m$, $\C Z_{\tau,\delta}^m$.

	We obtain the following time-discrete (in)equalities
	by the minimizing property of the functions $(u_\tau^m,z_\tau^m)$ with respect to the functional $\C F_\tau^m$ over $\C U_\tau^m\times\C Z_\tau^m$.
	\begin{lemma}
		The functions $(u_\tau^m,z_\tau^m)\in \C U_\tau^m\times\C Z_\tau^m$ for $m=0,\ldots,M$
		and, consequently, the piecewise constant and linear interpolations
		\begin{align*}
			&u_\tau,v_\tau\in L^\infty(0,T;\Hb),\;\widehat u_\tau,\widehat v_\tau\in W^{1,\infty}(0,T;H^2(\Omega;\R^n)),\\
			&z_\tau\in L^\infty(0,T;\Wa),\;\widehat z_\tau\in W^{1,\infty}(0,T;\Wa)
		\end{align*}
		satisfy
		\begin{itemize}
			\item[(i)]
				for all $\zeta\in H_{\Gamma_\mathrm{D}}^2(\Omega;\R^n)$ and for all $t\in(0,T)$:
				\begin{align}
				\label{eqn:discrMomentumBalance}
					\int_\Omega\partial_t\widehat v_\tau\cdot\zeta\dx+\int_\Omega W_{,e}(\e(u_\tau),z_\tau):\e(\zeta)\dx+\delta \langle Au_\tau,\zeta\rangle_{H^2}
						=\int_\Omega \ell_\tau\cdot\zeta\dx,
				\end{align}
			\item[(ii)]
				for all $\zeta\in \Wa$ and all $t\in(0,T)$ with $0\leq\zeta\leq z_\tau^-(t)$ a.e. in $\Omega$:
				\begin{align}
				\label{eqn:discrDamageIneq}
					0\leq\int_\Omega\Big(|\nabla z_\tau|^{p-2}\nabla z_\tau\cdot\nabla(\zeta-z_\tau)+(W_{,z}(\e(u_\tau),z_\tau)+f'(z_\tau)
					+\partial_t\widehat z_\tau)(\zeta-z_\tau)\Big)\dx.
				\end{align}
		\end{itemize}
	\end{lemma}
	\begin{proof}
		The minimizer $(u_\tau^m,z_\tau^m)$ fulfills the variational property
		\begin{align*}
			-\mathrm d_u \C F_\tau^m(u_\tau^m,z_\tau^m)&=0,\\
			-\mathrm d_z \C F_\tau^m(u_\tau^m,z_\tau^m)&\in N_{\C Z_\tau^m}(z_\tau^m),
		\end{align*}
		where $N_{\C Z_\tau^m}(z_\tau^m)$ denotes the normal cone to $\C Z_\tau^m$
		at $z_\tau^m$. Equivalently,
		\begin{align*}
			\int_\Omega\frac{u_\tau^m-2u_\tau^{m-1}+u_\tau^{m-2}}{\tau^2}\cdot\zeta\dx+\int_\Omega W_{,e}(\e(u_\tau^m),z_\tau^m):\e(\zeta)\dx+\delta \langle Au_\tau^m,\zeta\rangle_{H^2}
				=\int_\Omega \ell_\tau^m\cdot\zeta\dx
		\end{align*}
		for all $\zeta\in H_\GammaD^2(\Omega;\Rn)$ and
		\begin{align*}
			0\leq\int_\Omega\Big(|\nabla z_\tau^m|^{p-2}\nabla z_\tau^m\cdot\nabla(\zeta-z_\tau^m(t))+\Big(W_{,z}(\e(u_\tau^m),z_\tau^m)+f'(z_\tau^m)
			+\frac{z_\tau^m-z_\tau^{m-1}}{\tau}\Big)(\zeta-z_\tau^m(t))\Big)\dx.
		\end{align*}
		for all $\zeta\in \C Z_\tau^m$.
		\ep
	\end{proof}

	\begin{lemma}[A priori estimates]
	\label{lemma:aprioriDiscr}
	\hspace*{0.1em}
		\begin{itemize}
			\item[(i)]
				The following a priori estimates hold uniformly in $\tau$ and $\delta$:
				\begin{subequations}
				\begin{align}
				\label{eqn:aprioriU1}
					&\sqrt{\delta}\|u_{\tau,\delta}\|_{L^\infty(0,T;H^2(\Omega;\R^n))}\leq \sqrt{\delta} C\|u^0\|_{H^2}+C,\\
				\label{eqn:aprioriU2}
					&\|u_{\tau,\delta}\|_{L^\infty(0,T;H^1(\Omega;\R^n))}\leq \sqrt{\delta}C\|u^0\|_{H^2}+C,\\
				\label{eqn:aprioriU3}
					&\|\widehat u_{\tau,\delta}\|_{L^\infty(0,T;H^1(\Omega;\R^n))\cap W^{1,\infty}(0,T;L^2(\Omega;\R^n))}\leq \sqrt{\delta}C\|u^0\|_{H^2}+C,\\
				\label{eqn:aprioriV1}
					&\|v_{\tau,\delta}\|_{L^\infty(0,T;L^2(\Omega;\R^n))}\leq \sqrt{\delta}C\|u^0\|_{H^2}+C,\\
				\label{eqn:aprioriV2}
					&\|\widehat v_{\tau,\delta}\|_{L^{\infty}(0,T;L^2(\Omega;\R^n))\cap H^1(0,T;(H_{\Gamma_\mathrm{D}}^2(\Omega;\R^n))^*)}\leq \sqrt{\delta}C\|u^0\|_{H^2}+C.
				\end{align}
				\end{subequations}
			\item[(ii)]
				For fixed $\delta>0$, the following additional a priori estimates hold uniformly in $\tau$:
				\begin{subequations}
				\begin{align}
				\label{eqn:aprioriZ1}
					&\|z_{\tau,\delta}\|_{L^\infty(0,T;\Wa)}\leq C,\\
				\label{eqn:aprioriZ2}
					&\|\widehat z_{\tau,\delta}\|_{L^\infty(0,T;\Wa)\cap H^1(0,T;L^2(\Omega))}\leq C.
				\end{align}
				\end{subequations}
		\end{itemize}
	\end{lemma}
	\begin{remark}
	\label{remark:reg}
		The a priori estimates in (ii) are due to the $H^2$-regularization for $u$
		(the a priori bound \eqref{eqn:aprioriU1} to be more precise).
		Later on, it will also enable us to establish the total energy-dissipation inequality
		in the limit regime $\tau\searrow 0$.
		Then, the total energy-dissipation inequality will give a priori estimates of type (ii)
		uniformly in $\delta$.
	\end{remark}
	\textbf{Proof of Lemma \ref{lemma:aprioriDiscr}}
	
		\underline{To (i): Proof of the a priori estimates \eqref{eqn:aprioriU1}-\eqref{eqn:aprioriV2}:}
		
				Testing \eqref{eqn:discrMomentumBalance} with $u_\tau-u_\tau^--(b_\tau-b_\tau^-)$ and using the estimate
				\begin{align*}
					\int_\Omega\partial_t\widehat v_\tau\cdot(u_\tau-u_\tau^-)\dx&\geq \frac12\left\|v_\tau\right\|_{L^2}^2-\frac12\left\|v_\tau^-\right\|_{L^2}^2
				\end{align*}
				as well as the convexity estimates (note that $z_\tau^-\geq z_\tau$)
				\begin{align}
				\label{eqn:West}
					\int_\Omega W_{,e}(\e(u_\tau),z_\tau):\e(u_\tau-u_\tau^-)\dx
						&\geq\int_\Omega\bl W(\e(u_\tau),z_\tau)-W(\e(u_\tau^-),z_\tau^-)\br\dx\\
					\delta\langle Au_\tau,u_\tau-u_\tau^-\rangle_{H^2}
						&\geq \frac\delta2 \langle A u_\tau,u_\tau\rangle_{H^2}-\frac\delta2 \langle A u_\tau^-,u_\tau^-\rangle_{H^2},
				\end{align}
				yield
				\begin{align}
					&\frac12\left\|v_\tau(t)\right\|_{L^2}^2-\frac12\left\|v_\tau^-(t)\right\|_{L^2}^2
						+\frac\delta2 \langle A u_\tau(t),u_\tau(t)\rangle_{H^2}-\frac\delta2 \langle A u_\tau^-(t),u_\tau^-(t)\rangle_{H^2}\notag\\
					&+\int_\Omega\bl W(\e(u_\tau(t)),z_\tau(t))-W(\e(u_\tau^-(t)),z_\tau^-(t))\br\dx
						-\int_\Omega\partial_t \widehat v_\tau(t)\cdot\bl b_\tau(t)-b_\tau^-(t)\br\dx\notag\\
					&\qquad\leq \int_\Omega \ell_\tau(t)\cdot\bl u_\tau(t)-u_\tau^-(t)-(b_\tau(t)-b_\tau^-(t))\br\dx\notag\\
				\label{eqn:testEq1}
					&\qquad\quad+\int_\Omega W_{,e}(\e(u_\tau(t)),z_\tau(t)):\e(b_\tau(t)-b_\tau^-(t))\dx
						+\delta\langle A u_\tau(t),b_\tau(t)-b_\tau^-(t)\rangle_{H^2}.
				\end{align}
				The right-hand side can be estimated by Young's inequality as follows ($\eta>0$ denotes a freely chosen constant)
				\begin{align}
					\text{r.h.s.}
					\leq{}&
						C\tau\Big(\|\ell_\tau(t)\|_{L^2}^2+\|v_\tau(t)\|_{L^2}^2+\|\partial_t \widehat b_\tau(t)\|_{L^2}^2
						+\eta\|\e(u_\tau(t))\|_{L^2}^2+C_\eta\|\e(\partial_t \widehat b_\tau(t))\|_{L^2}^2\notag\\
				\label{eqn:rhsEst}
					&+\delta\langle A u_\tau(t),u_\tau(t)\rangle_{H^2}
						+\delta\langle A \partial_t \widehat b_\tau(t),\partial_t \widehat b_\tau(t)\rangle_{H^2}\Big).
				\end{align}
				Summing \eqref{eqn:testEq1} over the discrete time points $\tau,2\tau,\ldots,t_\tau$,
				using the estimate \eqref{eqn:rhsEst}, an $L^2(L^2)$-a priori bound for $\ell_\tau$
				and an $H^1(H^1)$-a priori bound for $\widehat b_\tau$,
				we obtain for small $\eta>0$
				\begin{align}
					&\frac12\left\|v_\tau(t)\right\|_{L^2}^2+\frac\delta2 \langle A u_\tau(t),u_\tau(t)\rangle_{H^2}+c\|\e(u_\tau(t))\|_{L^2}^2
						-\int_0^{t_\tau}\int_{\Omega}\partial_t \widehat v_\tau\cdot\partial_t \widehat b_\tau\dxs\notag\\
					&\qquad\leq \frac12\left\|v^0\right\|_{L^2}^2
						+\frac\delta2 \langle A u^0,u^0\rangle_{H^2}
						+\int_\Omega W(\e(u^0),z^0)\dx\notag\\
					&\qquad\quad+C\int_0^{t_\tau}\bl1+\|v_\tau(s)\|_{L^2}^2+\eta\|\e(u_\tau(t))\|_{L^2}^2+\delta\langle A u_\tau(s),u_\tau(s)\rangle_{H^2}+C_\eta\br\ds
				\label{eqn:aprioriTemp}
				\end{align}
				for all $t\in[0,T]$.
				The discrete integration by parts formula yields for all $t\in[0,T]$
				\begin{align}
					\int_0^{t_\tau}\int_{\Omega}\partial_t \widehat v_\tau\cdot\partial_t \widehat b_\tau\dxs
						={}&\int_\Omega v_\tau(t)\cdot\partial_t \widehat b_\tau(t)\dx
						-\int_\Omega v^0\cdot\partial_t\widehat b_\tau(0)\dx\notag\\
					&-\int_0^{t_\tau}\int_\Omega v_\tau^-(s)\cdot\frac{\partial_t\widehat b_\tau(s)-\partial_t\widehat b_\tau(s-\tau)}{\tau}\dxs.
				\label{eqn:discrIntegrByParts}
				\end{align}
				Applying \eqref{eqn:discrIntegrByParts} to \eqref{eqn:aprioriTemp} and using an $L^2(L^2)$- a priori bound for
				the second discrete time-derivative of $\widehat b_\tau$, we eventually obtain for all $t\in[0,T]$
				\begin{align*}
				\begin{split}
					&\frac12\left\|v_\tau(t)\right\|_{L^2}^2+\frac\delta2 \langle A u_\tau(t),u_\tau(t)\rangle_{H^2}+c\|\e(u_\tau(t))\|_{L^2}^2\\
					&\quad\leq
						C\Big(1+\delta\|u^0\|_{H^2}^2+\int_0^{t_\tau}\bl\|v_\tau(s)\|_{L^2}^2+\eta\|\e(u_\tau(s))\|_{L^2}^2+\delta\langle A u_\tau(s),u_\tau(s)\rangle_{H^2}\br\ds\Big)\\
					&\quad\quad-\int_0^{t_\tau}\int_\Omega v_\tau^-(s)\cdot\frac{\partial_t\widehat b_\tau(s)-\partial_t\widehat b_\tau(s-\tau)}{\tau}\dxs\\
					&\quad\leq
						C\Big(1+\delta\|u^0\|_{H^2}^2+\int_0^{t_\tau}\bl\|v_\tau(s)\|_{L^2}^2+\|v_\tau^-(s)\|_{L^2}^2+\eta\|\e(u_\tau(s))\|_{L^2}^2
						+\delta\langle A u_\tau(s),u_\tau(s)\rangle_{H^2}\br\ds\Big).
				\end{split}
				\end{align*}
				We conclude by a discrete version of Gronwall's inequality
				\eqref{eqn:aprioriU1}-\eqref{eqn:aprioriV1}
				and
				$$
					\|\widehat v_{\tau,\delta}\|_{L^{\infty}(0,T;L^2(\Omega;\R^n))}<\sqrt{\delta}C\|u^0\|_{H^2}+C
				$$
				with the help of Korn's inequality.
				By using these a priori estimates, a comparison argument in \eqref{eqn:discrMomentumBalance} shows
				\begin{align*}
					\|\partial_t\widehat v_\tau\|_{L^2(0,T;(H_{\Gamma_\mathrm{D}}^2(\Omega;\R^n))^*)}^2
					={}&\int_0^T\Big|\sup_{\|\zeta\|_{H_{\GammaD}^2}=1}\langle\partial_t\widehat v_\tau(t),\zeta\rangle\Big|^2\dt\\
					\leq{}&\delta\|u_\tau(t)\|_{L^2(0,T;H^2(\Omega;\Rn))}^2+\int_0^T\int_\Omega\Big(|W_{,e}(\e(u_\tau),z_\tau)|^2
							+|\ell_\tau|^2\Big)\dxt\\
					\leq{}&\delta C\|u^0\|_{H^2}^2+C.
				\end{align*}
				Thus \eqref{eqn:aprioriV2} is proven.
		
		\underline{To (ii): Proof of the a priori estimates \eqref{eqn:aprioriZ1} and \eqref{eqn:aprioriZ2}:}
		
				Testing \eqref{eqn:discrDamageIneq} with $z_\tau$
				which is possible due to $0\leq z_\tau\leq z_\tau^-$ a.e. in $\Omega\times(0,T)$
				and using the $|\cdot|^p$-convexity estimate
				$$
					\int_\Omega |\nabla z_\tau|^{p-2}\nabla z_\tau\cdot\nabla(z_\tau-z_\tau^-)\dx
						\geq\frac 1p\|\nabla z_\tau\|_{L^p}^p-\frac1p\|\nabla z_\tau^-\|_{L^p}^p
				$$
				yield
				\begin{align}
					&\frac 1p\|\nabla z_\tau(t)\|_{L^p}^p-\frac1p\|\nabla z_\tau^-(t)\|_{L^p}^p
						+\tau\left\|\partial_t\widehat z_\tau(t)\right\|_{L^2}^2\notag\\
				\label{eqn:testEq2}
					&\qquad\leq
						\int_\Omega\bl W_{,z}(\e(u_\tau(t)),z_\tau(t))+f'(z_\tau(t))\br(z_\tau^-(t)-z_\tau(t))\dx.
				\end{align}
				Thus, by Young's inequality ($\eta>0$ denotes a freely chosen constant)
				\begin{align*}
					&\frac 1p\|\nabla z_\tau(t)\|_{L^p}^p-\frac1p\|\nabla z_\tau^-(t)\|_{L^p}^p
						+\tau\left\|\partial_t\widehat z_\tau(t)\right\|_{L^2}^2\\
					&\qquad\leq\tau\eta\left\|\partial_t\widehat z_\tau(t)\right\|_{L^2}^2+\tau C_\eta\|W_{,z}(\e(u_\tau(t)),z_\tau(t))+f'(z_\tau(t))\|_{L^2}^2.
				\end{align*}
				Summing over the discrete time points $\tau,2\tau,\ldots, t_\tau$ and using the constraint $0\leq z_\tau\leq 1$
				a.e. in $\Omega\times(0,T)$
				as well as the a priori bound $\|u_\tau\|_{L^\infty(0,T;W^{1,4}(\Omega;\R^n))}<C$ (which follows from \eqref{eqn:aprioriU1} for fixed $\delta>0$
				and the continuous embedding $H^2(\Omega;\Rn)\hookrightarrow W^{1,4}(\Omega;\Rn)$), we end up with
				\eqref{eqn:aprioriZ1} and \eqref{eqn:aprioriZ2}.
		\ep
	

	The a priori estimates in Lemma \ref{lemma:aprioriDiscr} give rise to the following convergence properties
	which follows from standard weak and Aubin-Lions type compactness results \cite{Simon}.
	\begin{lemma}
	\label{lemma:discrConvergence}
		There exist functions
		\begin{align*}
			&u\in L^\infty(0,T;H^2(\Omega;\R^n))\cap W^{1,\infty}(0,T;L^2(\Omega;\R^n))\cap H^2(0,T;(H_{\Gamma_\mathrm{D}}^2(\Omega;\R^n))^*),\\
			&z\in L^\infty(0,T;\Wa)\cap H^1(0,T;L^2(\Omega))
		\end{align*}
		satisfying \eqref{eqn:initialBoundary}, \eqref{eqn:weakVarIneq1b} and \eqref{eqn:weakVarIneq2b}
		and a subsequence $\tau_k\searrow 0$ as $k\nearrow\infty$ such that
		\begin{subequations}
		\begin{align}
		\label{eqn:convU1}
			u_{\tau_k},u_{\tau_k}^-&\to u&&\text{ weakly-star in }L^\infty(0,T;H^2(\Omega;\R^n)),\\
		\label{eqn:convU2}
				&&&\text{ strongly in }L^q(0,T;W^{1,s}(\Omega;\R^n))\text{ for every }q\in[1,\infty),\;1\leq s<2^*,\\
		\label{eqn:convU3}
			\widehat u_{\tau_k}&\to u&&\text{ weakly-star in }L^\infty(0,T;H^2(\Omega;\R^n))\cap W^{1,\infty}(0,T;L^2(\Omega;\Rn)),\\
		\label{eqn:convU4}
				&&&\text{ strongly in }L^q(0,T;W^{1,s}(\Omega;\R^n))\text{ for every }q\in[1,\infty),\;1\leq s<2^*,\\
		\label{eqn:convV1}
			v_{\tau_k},v_{\tau_k}^-&\to \partial_{t}u&&\text{ weakly-star in }L^\infty(0,T;L^2(\Omega;\R^n)),\\
		\label{eqn:convV2}
			\widehat v_{\tau_k}&\to \partial_{t}u&&\text{ weakly-star in }L^\infty(0,T;L^2(\Omega;\R^n))\text{ and weakly in }H^{1}(0,T;(H_\GammaD^2(\Omega;\Rn))^*),\\
		\label{eqn:convZ1}
			z_{\tau_k},z_{\tau_k}^-&\to z&&\text{ weakly-star in }L^\infty(0,T;\Wa),\\
		\label{eqn:convZ2}
				&&&\text{ strongly in }L^q(0,T;L^\infty(\Omega))\text{ for every }q\in[1,\infty),\\
		\label{eqn:convZ3}
			\widehat z_{\tau_k}&\to z&&\text{ weakly-star in }L^\infty(0,T;\Wa)\text{ and weakly in }H^1(0,T;L^2(\Omega)),\\
		\label{eqn:convZ4}
				&&&\text{ strongly in } C(\ol{\Omega}\times[0,T])
		\end{align}
		\end{subequations}
		as $k\nearrow\infty$ for fixed $\delta>0$.
		The constant $2^*$ denotes the critical Sobolev exponent.
		Moreover, we obtain the following convergence properties of the data
		\begin{subequations}
		\begin{align}
		\label{eqn:convEll}
			\hspace*{2em}\ell_\tau&\to\ell
				&&\text{strongly in }L^2(0,T;L^2(\Omega;\Rn)),\hspace*{10.9em}\\
		\label{eqn:convB1}
			\widehat b_{\tau_k}&\to b
				&&\text{strongly in }H^1(0,T;H^2(\Omega;\Rn)),\\
		\label{eqn:convB2}
			\frac{\partial_t\widehat b_{\tau_k}-\partial_t\widehat b_{\tau_k}(\cdot-\tau_k)}{\tau_k} &\to \partial_{tt}b
				&&\text{strongly in }L^2(0,T;H^2(\Omega;\Rn)).
		\end{align}
		\end{subequations}
	\end{lemma}
	\begin{proof}
		\underline{To \eqref{eqn:convEll}-\eqref{eqn:convB2}:}
		
		We set $X:=H^2(\Omega;\Rn)$.
		By exploiting the fundamental theorem of calculus for functions with values in $X$ and the assumed Lipschitz continuity of $\partial_{tt}b$ in time,
		a straightforward calculation shows
		\begin{align*}
			\int_0^T\left\|\frac{\partial_t\widehat b_{\tau_k}(t)-\partial_t\widehat b_{\tau_k}(t-\tau_k)}{\tau_k}-\partial_{tt}b(t)\right\|_{X}^2\dt
			={}&\frac{1}{\tau_k^4}\int_0^T\left\|\int_{t_{\tau_k}^-}^{t_{\tau_k}}\int_{s-\tau_k}^s\big(\partial_{tt}b(\iota)-\partial_{tt}b(t)\big)\diota\ds\right\|_{X}^2\dt\\
			\leq{}&\frac{1}{\tau_k^4}\int_0^T\left(\int_{t_{\tau_k}^-}^{t_{\tau_k}}\int_{s-\tau_k}^s\big\|\partial_{tt}b(\iota)-\partial_{tt}b(t)\big\|_{X}\diota\ds\right)^2\dt\\
			\leq{}&\frac{1}{\tau_k^4}\int_0^T\bigg(\int_{t_{\tau_k}^-}^{t_{\tau_k}}\int_{s-\tau_k}^sC\underbrace{|\iota-t|}_{\leq 2\tau_k}\diota\ds\bigg)^2\dt\\
			={}&4C^2\tau_k^2T\to 0
		\end{align*}
		as $k\nearrow\infty$. Thus \eqref{eqn:convB2} is shown. The properties \eqref{eqn:convEll} and \eqref{eqn:convB1} follow by similar reasoning.
		
		\underline{To \eqref{eqn:convU1}-\eqref{eqn:convZ4}:}
		
		Standard weak and weak-star compactness results applied to the a priori estimates in Lemma \ref{lemma:aprioriDiscr}
		reveal existence of functions
		\begin{align*}
			&u,u^-\in L^\infty(0,T;H^2(\Omega;\R^n)),\\
			&\widehat u\in L^\infty(0,T;H^2(\Omega;\R^n))\cap W^{1,\infty}(0,T;L^2(\Omega;\R^n)),\\
			&v,v^-\in L^{\infty}(0,T;L^2(\Omega;\R^n)),\\
			&\widehat v\in L^{\infty}(0,T;L^2(\Omega;\R^n))\cap H^{1}(0,T;(H_{\Gamma_\mathrm{D}}^2(\Omega;\R^n))^*),\hspace*{2.8em}\\
			&z,z^-\in L^\infty(0,T;\Wa),\\
			&\widehat z\in L^\infty(0,T;\Wa)\cap H^1(0,T;L^2(\Omega))
		\end{align*}
		satisfying \eqref{eqn:initialBoundary}, \eqref{eqn:weakVarIneq1b} and \eqref{eqn:weakVarIneq2b}
		and subsequences indexed by $\tau_k$ such that
		\begin{align*}
			&u_{\tau_k}\to u&&\text{ weakly-star in }L^\infty(0,T;H^2(\Omega;\R^n)),\\
			&u_{\tau_k}^-\to u^-&&\text{ weakly-star in }L^\infty(0,T;H^2(\Omega;\R^n)),\\
			&\widehat u_{\tau_k}\to \widehat u&&\text{ weakly-star in }L^\infty(0,T;H^2(\Omega;\R^n))\cap W^{1,\infty}(0,T;L^2(\Omega;\R^n)),\\
			&v_{\tau_k}\to v&&\text{ weakly-star in }L^{\infty}(0,T;L^2(\Omega;\R^n)),\\
			&v_{\tau_k}^-\to v^-&&\text{ weakly-star in }L^{\infty}(0,T;L^2(\Omega;\R^n)),\\
			&\widehat v_{\tau_k}\to \widehat v&&\text{ weakly-star in }L^{\infty}(0,T;L^2(\Omega;\R^n))\text{ and weakly in }H^1(0,T;(H_{\Gamma_\mathrm{D}}^2(\Omega;\R^n))^*),\\
			&z_{\tau_k}\to z&&\text{ weakly-star in }L^\infty(0,T;\Wa),\\
			&z_{\tau_k}^-\to z^-&&\text{ weakly-star in }L^\infty(0,T;\Wa),\\
			&\widehat z_{\tau_k}\to\widehat z&&\text{ weakly-star in } L^\infty(0,T;\Wa)\text{ and weakly in }H^1(0,T;L^2(\Omega))
		\end{align*}
		as $k\nearrow\infty$.
		Taking into account
		$$
			u_{\tau_k}-u_{\tau_k}^-={\tau_k}\partial_t\widehat u_{\tau_k}\to 0\text{ strongly in }L^{\infty}(0,T;L^2(\Omega;\R^n)),
		$$
		we obtain $u=u^-=\widehat u$.
		Analogously, we get $v=v^-=\widehat v$ and $z=z^-=\widehat z$.
		The identity $\partial_t\widehat u_{\tau_k}=v_{\tau_k}$ implies $\partial_t u=v$.
		
		Therefore, we obtain
		\begin{align*}
			&u\in L^\infty(0,T;H^2(\Omega;\R^n))\cap W^{1,\infty}(0,T;L^2(\Omega;\R^n))\cap H^2(0,T;(H_{\Gamma_\mathrm{D}}^2(\Omega;\R^n))^*),\\
			&z\in L^\infty(0,T;\Wa)\cap H^1(0,T;L^2(\Omega))
		\end{align*}
		such that \eqref{eqn:convU1}, \eqref{eqn:convU3}, \eqref{eqn:convV1}, \eqref{eqn:convV2}, \eqref{eqn:convZ1} and \eqref{eqn:convZ3}
		is satisfied for the subsequence $\{\tau_k\}_{k\in\N}$.
		
		Compactness arguments (in particular using the embedding $W^{1,p}(\Omega)\hookrightarrow C(\ol\Omega)$ valid for $p\in(n,\infty)$
		and Aubin-Lions type results; see \cite{Simon}) show \eqref{eqn:convU2}, \eqref{eqn:convU4}, \eqref{eqn:convZ2} and \eqref{eqn:convZ4}.
		\ep
	\end{proof}
	\begin{remark}
	\label{remark:pointwiseConv}
		By choosing further subsequences (we omit the additional subscript), we also obtain for fixed $\delta>0$
		\begin{align*}
			&u_{\tau_k},u_{\tau_k}^-,\widehat u_{\tau_k}\to u&&\text{pointwise a.e. in }\Omega\times(0,T),\\
			&\nabla u_{\tau_k},\nabla u_{\tau_k}^-,\nabla \widehat u_{\tau_k}\to\nabla u&&\text{pointwise a.e. in }\Omega\times(0,T),\\
			&z_{\tau_k},z_{\tau_k}^-,\widehat z_{\tau_k}\to z&&\text{pointwise a.e. in }\Omega\times(0,T)
		\end{align*}
		and for a.e. $t\in(0,T)$
		\begin{align*}
			\hspace*{1em}&u_{\tau_k}(t),u_{\tau_k}^-(t),\widehat u_{\tau_k}(t)\to u(t)&&\text{weakly in }H^2(\Omega;\Rn),\hspace*{4.0em}\\
			&z_{\tau_k}(t),z_{\tau_k}^-(t),\widehat z_{\tau_k}(t)\to z(t)&&\text{weakly in }W^{1,p}(\Omega)
		\end{align*}
		as $k\nearrow\infty$.
	\end{remark}
	Strong convergence of  $\nabla z_{\tau_k}\to\nabla z$ in $L^p(\Omega\times(0,T))$ can
	be shown by a subtle approximation argument introduced in \cite{WIAS1520}.
	\begin{lemma}
	\label{lemma:strongConvZ}
		There exists a subsequence of $\tau_k$ (omitting the additional subscript) such that
		$z_{\tau_k}\to z$ in $L^p(0,T;W^{1,p}(\Omega))$ as $k\nearrow\infty$.
	\end{lemma}
	\begin{proof}
		We apply the preliminary result cited in Lemma \ref{lemma:approximation} based on the work \cite{WIAS1520}.
		By Remark \ref{remark:pointwiseConv}, the sequence $\{z_{\tau_k}\}_{k\in\N}$ fulfills the assumption.
		According to Lemma \ref{lemma:approximation} there exists an approximation sequence $\{\zeta_{\tau_k}\}\subseteq L^\infty(0,T;W_+^{1,p}(\Omega))$ with
		the properties
		\begin{align}
		\label{eqn:zApprox}
			&\zeta_{\tau_k}\rightarrow z\hspace*{4em}&&\text{strongly in }L^p(0,T;W^{1,p}(\Omega)),\hspace*{8em}\\
		\label{eqn:zApprox2}
			&0\leq \zeta_{\tau_k}\leq z_{\tau_k}^-&&\text{pointwise a.e. in }\Omega\times(0,T)\text{ for all }k\in\N.
		\end{align}
		We omit the subscript $k$ for notational convenience.
		Property \eqref{eqn:zApprox2}
		enables us to test \eqref{eqn:discrDamageIneq} with $\zeta_\tau$. Integration over the time variable shows
		$$
			\int_0^T\int_{\Omega}|\nabla z_\tau|^{p-2}\nabla z_\tau\cdot\nabla(z_\tau-\zeta_\tau)\dxt
				\leq\int_0^T\int_{\Omega}\bl W_{,z}(\e(u_\tau),z_\tau)+f'(z_\tau)+\partial_t\widehat z_\tau\br(\zeta_\tau-z_\tau)\dxt.
		$$
		A uniform $p$-monotonicity argument and the above estimate show ($c>0$ is a constant)
		\begin{align}
			&c\|\nabla z-\nabla z_\tau\|_{L^p(0,T;L^p(\Omega;\Rn))}^p\notag\\
				&\qquad\leq \int_0^T\int_{\Omega}\bl|\nabla z|^{p-2}\nabla z-|\nabla z_\tau|^{p-2}\nabla z_\tau\br\cdot\nabla(z-z_\tau)\dxt\notag\\
				&\qquad=\int_0^T\int_{\Omega}|\nabla z|^{p-2}\nabla z\cdot\nabla(z-z_\tau)\dxt
					+\int_0^T\int_{\Omega} |\nabla z_\tau|^{p-2}\nabla z_\tau\cdot\nabla(z_\tau-\zeta_\tau)\dxt\notag\\
					&\qquad\quad+\int_0^T\int_{\Omega} |\nabla z_\tau|^{p-2}\nabla z_\tau\cdot\nabla(\zeta_\tau-z)\dxt\notag\\
			&\qquad\leq
				\int_0^T\int_{\Omega}\bl W_{,z}(\e(u_\tau),z_\tau)+f'(z_\tau)+\partial_t\widehat z_\tau\br(\zeta_\tau-z_\tau)\dxt\notag\\
				&\qquad\quad+\int_0^T\int_{\Omega}|\nabla z|^{p-2}\nabla z\cdot\nabla(z-z_\tau)\dxt
				+\int_0^T\int_{\Omega} |\nabla z_\tau|^{p-2}\nabla z_\tau\cdot\nabla(\zeta_\tau-z)\dxt.
		\label{eqn:zEst}
		\end{align}
		In the following, we prove that every term on the right hand side converges to $0$ as $\tau\searrow 0$.
		\begin{itemize}
			\item[$\bullet$]
				The first integral on the r.h.s of \eqref{eqn:zEst} can be estimated as follows
				\begin{align}
					&\int_0^T\int_{\Omega}\bl W_{,z}(\e(u_\tau),z_\tau)+f'(z_\tau)+\partial_t\widehat z_\tau\br(\zeta_\tau-z_\tau)\dxt\notag\\
						&\qquad\leq
							\big\|W_{,z}(\e(u_\tau),z_\tau)+f'(z_\tau)\big\|_{L^2(0,T;L^1(\Omega))}\|\zeta_\tau-z_\tau\|_{L^2(0,T;L^\infty(\Omega))}\notag\\
				\label{eqn:damageIneq}
						&\qquad\quad+\big\|\partial_t\widehat z_\tau\big\|_{L^2(\Omega_T)}\big\|\zeta_\tau-z_\tau\big\|_{L^2(\Omega_T)}.
				\end{align}
				By using the boundedness of $u_\tau$ in $L^\infty(0,T;H^1(\Omega;\R^n))$,
				boundedness of $\partial_t \widehat z_\tau$ in $L^2(0,T;L^2(\Omega))$ (see Lemma \ref{lemma:aprioriDiscr}),
				boundedness $z_\tau\in[0,1]$ a.e. in $\Omega_T$
				and the convergence properties \eqref{eqn:convZ2} and \eqref{eqn:zApprox},
				we obtain convergence to $0$ as of the two summands on the right
				hand side of \eqref{eqn:damageIneq}.
			\item[$\bullet$]
				Due to the convergence \eqref{eqn:convZ1}, the second integral on the r.h.s. of \eqref{eqn:zEst} converges to $0$ as $\tau\searrow 0$.
			\item[$\bullet$]
				We estimate the third integral on the r.h.s. of \eqref{eqn:zEst} by H\"older's inequality:
				\begin{align*}
					\int_0^T\int_\Omega |\nabla z_\tau|^{p-2}\nabla z_\tau\cdot\nabla(\zeta_\tau-z)\dxt
						\leq \|\nabla z_\tau\|_{L^p(0,T;L^p(\Omega))}^{p-1}\|\nabla(\zeta_\tau-z)\|_{L^p(0,T;L^p(\Omega))}.
				\end{align*}
				Because of the boundedness property \eqref{eqn:aprioriZ1} and the strong convergence property \eqref{eqn:zApprox},
				we obtain convergence to $0$ of the integral term above.
		\end{itemize}
		Combing the convergence result $\nabla z_{\tau}\to \nabla z$ strongly in $L^p(0,T;L^p(\Omega;\Rn))$ as $\tau\searrow 0$
		with Lemma \ref{lemma:discrConvergence}, the claim follows.
		\ep
	\end{proof}
	
	The notion of weak solutions as given in Definition \ref{def:regWeakSolution} requires
	the validity of the total energy-dissipation inequality.
	However, in this discrete setting, we are only able to prove an approximate version of this inequality.
	But the $H^2$-regularization enables us to recover the postulated total energy-dissipation inequality in the limit
	$\tau\searrow 0$ as already indicated in Remark \ref{remark:reg}.
	\begin{lemma}
	\label{lemma:discrEI}
		For a.e. $t\in(0,T)$ the approximate energy-dissipation inequality
		\begin{align}
			\C F_\tau(t)+\C K_\tau(t)+\C D_\tau(0,t)+\C E_\tau(0,t)
				\leq \C F(0)+\C K(0)+\C W_{ext}^\tau(0,t)
		\label{eqn:discrEI}
		\end{align}
		with
		\begin{align*}
			&\C F_\tau(t):=\int_\Omega\bl\frac 1p|\nabla z_\tau(t)|^p+W(\e(u_\tau(t)),z_\tau(t))+f(z_\tau(t))\br\dx
				+\frac\delta2 \langle A u_\tau(t),u_\tau(t)\rangle_{H^2},\\
			&\C K_\tau(t):=\int_\Omega\frac12|v_\tau(t)|^2\dx,\\
			&\C D_\tau(0,t):=\int_0^{t_\tau}\int_\Omega|\partial_t\widehat z_\tau|^2\dxs,\\
			&\C W_{ext}^\tau(0,t):=
				\int_0^{t_\tau}\int_\Omega W_{,e}(\e(u_\tau),z_\tau):\e(\partial_t\widehat b_\tau)\dxs
				-\int_0^{t_\tau}\int_\Omega v_\tau^-(s)\cdot\frac{\partial_t\widehat b_\tau(s)-\partial_t\widehat b_\tau(s-\tau)}{\tau}\dxs\\
				&\qquad\qquad\qquad+\int_\Omega v_\tau(t)\cdot\partial_t \widehat b_\tau(t)\dx
				-\int_\Omega v^0\cdot\partial_t\widehat b_\tau(0)\dx
				+\int_0^{t_\tau}\int_\Omega \ell_\tau\cdot\bl \partial_t\widehat u_\tau-\partial_t\widehat b_\tau\br\dxs\\
				&\qquad\qquad\qquad+\delta\int_0^{t_\tau}\langle A u_\tau(s),\partial_t\widehat b_\tau(s)\rangle_{H^2}\ds
		\end{align*}
		and the ``error term''
		\begin{align*}
			\C E_\tau(0,t):={}&
				\int_0^{t_\tau}\int_\Omega\frac 12\frac{h(z_\tau^-)-h(z_\tau)}{\tau}\mathbf C\e(u_\tau^-):\e(u_\tau^-)\dxs\\
				&+\int_0^{t_\tau}\int_\Omega W_{,z}(\e(u_\tau),z_\tau)\,\partial_t \widehat z_\tau\dxs\\
				&-\int_0^{t_\tau}\int_\Omega\frac{f(z_\tau)-f(z_\tau^-)}{\tau}\dxs
				+\int_0^{t_\tau}\int_\Omega f'(z_\tau)\,\partial_t \widehat z_\tau\dxs
		\end{align*}
		holds.
	\end{lemma}
	\begin{proof}
		By employing the estimate (which is slightly sharper than the convexity estimate \eqref{eqn:West})
		\begin{align*}
			&\int_\Omega W_{,e}(\e(u_\tau),z_\tau):\e(u_\tau-u_\tau^-)\dx\notag\\
			&\qquad\geq\int_\Omega\bl W(\e(u_\tau),z_\tau)-W(\e(u_\tau^-),z_\tau^-)\br\dx\\
			&\qquad\quad+\int_\Omega\frac 12\bl h(z_\tau^-)-h(z_\tau)\br\mathbf C\e(u_\tau^-):\e(u_\tau^-)\dx,
		\end{align*}
		we obtain by testing \eqref{eqn:discrMomentumBalance} with
		$u_\tau-u_\tau^--(b_\tau-b_\tau^-)$ (cf. \eqref{eqn:testEq1}):
		\begin{align}
			&\frac12\left\|v_\tau(t)\right\|_{L^2}^2-\frac12\left\|v_\tau^-(t)\right\|_{L^2}^2
				+\frac\delta2 \langle A u_\tau(t),u_\tau(t)\rangle_{H^2}-\frac\delta2 \langle A u_\tau^-(t),u_\tau^-(t)\rangle_{H^2}\notag\\
			&+\int_\Omega\bl W(\e(u_\tau(t)),z_\tau(t))-W(\e(u_\tau^-(t)),z_\tau^-(t))\br\dx
				-\int_\Omega\partial_t \widehat v_\tau(t)\cdot\bl b_\tau(t)-b_\tau^-(t)\br\dx\notag\\
			&+\int_\Omega\frac 12\bl h(z_\tau^-(t))-h(z_\tau(t))\br\mathbf C\e(u_\tau^-(t)):\e(u_\tau^-(t))\dx\notag\\
			&\qquad\leq \int_\Omega \ell_\tau(t)\cdot\bl u_\tau(t)-u_\tau^-(t)-(b_\tau(t)-b_\tau^-(t))\br\dx\notag\\
		\label{eqn:testMom2}
			&\qquad\quad+\int_\Omega W_{,e}(\e(u_\tau(t)),z_\tau(t)):\e(b_\tau(t)-b_\tau^-(t))\dx
				+\delta\langle A u_\tau(t),b_\tau(t)-b_\tau^-(t)\rangle_{H^2}.
		\end{align}
		By testing \eqref{eqn:discrDamageIneq} with $z_\tau$, we obtain
		\begin{align}
			&\frac 1p\|\nabla z_\tau(t)\|_{L^p}^p-\frac1p\|\nabla z_\tau^-(t)\|_{L^p}^p
				+\tau\left\|\partial_t\widehat z_\tau(t)\right\|_{L^2}^2\notag\\
		\label{eqn:testDamage2}
			&\qquad\leq
				\int_\Omega\bl W_{,z}(\e(u_\tau(t)),z_\tau(t))+f'(z_\tau(t))\br(z_\tau^-(t)-z_\tau(t))\dx.
		\end{align}
		as in the proof of Lemma \ref{lemma:aprioriDiscr}.
		Adding the estimates \eqref{eqn:testMom2} and \eqref{eqn:testDamage2},
		summing over the discrete time points and
		taking into account formula \eqref{eqn:discrIntegrByParts}
		yields \eqref{eqn:discrEI}.
		\ep
	\end{proof}

	We are now in the position to establish the equalities and inequalities of the weak formulation of Definition \ref{def:regWeakSolution}
	by passing $\tau\searrow 0$.
	As before, we omit the subscript $k$ in the sequence $\{\tau_k\}_{k\in\N}$.
	
	\textbf{Proof of Proposition \ref{proposition:regSystem}}
	
	The functions $u$ and $z$ from Lemma \ref{lemma:discrConvergence} already satisfy
	\eqref{eqn:initialBoundary}, \eqref{eqn:weakVarIneq1b} and \eqref{eqn:weakVarIneq2b}.
	It remains to show \eqref{eqn:weakVarIneq2a}, \eqref{eqn:weakDamageLaw2}, \eqref{eqn:weakRegMomentumBalance} and \eqref{eqn:weakRegEnergyInequality}
	from Definition \ref{def:regWeakSolution}.
	\begin{itemize}
		\item[]\hspace*{-2em}\underline{To \eqref{eqn:weakRegMomentumBalance}:}
			We find for all $\zeta\in H_{\Gamma_\mathrm{D}}^2(\Omega;\R^n)$
			$$
				\int_\Omega\partial_t\widehat v_\tau(t)\cdot\zeta\dx=\langle\partial_t\widehat v_\tau(t),\zeta\rangle_{H_\GammaD^2}
			$$
			by using the canonical embedding $L^2(\Omega;\R^n)\hookrightarrow (H_{\Gamma_\mathrm{D}}^2(\Omega;\R^n))^*$.
			Keeping this identity in mind, integrating \eqref{eqn:discrMomentumBalance} over time from $t=0$ to $t=T$ and passing to the limit $\tau\searrow 0$ for
			a subsequence by using the convergence properties in Lemma \ref{lemma:discrConvergence}, we obtain
			a time-integrated vesion of \eqref{eqn:weakRegMomentumBalance}.
			Then, switching back to an ``a.e. in $t$''-formulation shows \eqref{eqn:weakRegMomentumBalance}.
		\item[]\hspace*{-2em}\underline{To \eqref{eqn:weakVarIneq2a} and \eqref{eqn:weakDamageLaw2}:}
			The limit analysis for these equations are performed in two steps and makes use of the approximation technique cited in Lemma \ref{lemma:approximation}
			and the extension result cited in Lemma \ref{lemma:varProp}.
			\begin{itemize}
				\item[Step 1:]
					Let $\zeta\in L^\infty(0,T;W_-^{1,p}(\Omega))$ with $\{\zeta=0\}\supseteq\{z=0\}$ (see \eqref{eqn:setInclusion}).
					By Lemma \ref{lemma:approximation}, we obtain a sequence $\{\zeta_{\tau_k}\}\subseteq L^\infty(0,T;W_-^{1,p}(\Omega))$ (we omit $k$)
					and constants $\nu_{\tau,t}>0$ with the properties:
					\begin{subequations}
					\begin{align}
					\label{eqn:approxProp1}
						&\zeta_\tau\rightarrow\zeta&&\text{ strongly in }L^p(0,T;W^{1,p}(\Omega)),\\
					\label{eqn:approxProp2}
						&0\geq \nu_{\tau,t} \zeta_\tau(t)\geq -z_\tau(t)\text{ in }\Omega&&\text{ for a.e. }t\in(0,T)\text{ and all }\tau>0.
					\end{align}
					\end{subequations}
					The property \eqref{eqn:approxProp2} and $z_\tau\leq z_\tau^-$ holding pointwise a.e. in $\Omega\times(0,T)$, we also find for a.e.~$t\in(0,T)$
					$$
						0\leq\nu_{\tau,t}\zeta_\tau(t)+z_\tau(t)\leq z_\tau^-(t)\text{ a.e. in }\Omega.
					$$
					In consequence, for a.e.~$t\in(0,T)$, we can test \eqref{eqn:discrDamageIneq} with $\nu_{\tau,t}\zeta_\tau(t)+z_\tau(t)$ and obtain
					\begin{align*}
						&\nu_{\tau,t}\int_\Omega\Big(|\nabla z_\tau(t)|^{p-2}\nabla z_\tau(t)\cdot\nabla\zeta_\tau(t)+(W_{,z}(\e(u_\tau(t)),z_\tau(t))+f'(z_\tau(t))
						+\partial_t\widehat z_\tau(t))\zeta_\tau(t)\Big)\dx\\
						&\qquad\geq 0.
					\end{align*}
					We divide this inequality by the positive constant $\nu_{\tau,t}$ and integrate over the time interval $[0,T]$.
					The time-integration is necessary to exploit the weak convergence property for $\partial_t\widehat z_\tau(t))$ in $L^2(0,T;L^2(\Omega))$.
					More precisely, we use the convergence properties in Lemma \ref{lemma:discrConvergence}, Remark \eqref{remark:pointwiseConv} and
					Lemma \ref{lemma:strongConvZ} to pass to the limit $\tau\searrow 0$ for a subsequence and end up with
					\begin{align*}
						0\leq\int_0^T\int_{\Omega}\bl|\nabla z|^{p-2}\nabla z\cdot\nabla\zeta+(W_{,z}(\e(u),z)+f'(z)+\partial_t z)\zeta\br\dxt.
					\end{align*}
					In particular, we get an a.e. in time $t$ formulation.
				\item[Step 2:]
					We may apply Lemma \ref{lemma:varProp} to the above variational inequality.
					Then, we obtain for all $\zeta\in W_-^{1,p}(\Omega)$ the inequality
					\begin{align}
					\label{eqn:preVI}
						0\leq\int_0^T\int_{\Omega}\bl \partial_{t}z\zeta+|\nabla z|^{p-2}\nabla z\cdot\nabla\zeta+\big(W_{,z}(\e(u),z)+f'(z)+\widehat \xi\big)\zeta\br\dxt
					\end{align}
					with $\widehat \xi\in L^1(\Omega\times(0,T))$ given by
					\begin{align*}
						\widehat \xi=-\chi_{\{z=0\}}\mathrm{max}\Big\{0,\partial_{t}z+W_{,z}(e(u),z)+f'(z)\Big\}.
					\end{align*}
					Due to $\partial_{t}z\leq 0$ a.e. in $\Omega_T$,
					we may replace $\widehat\xi$ by $\xi\in L^1(0,T;L^1(\Omega))$ in \eqref{eqn:preVI}, where $\xi$ is given by
					\begin{align*}
						\xi=-\chi_{\{z=0\}}\mathrm{max}\Big\{0,W_{,z}(e(u),z)+f'(z)\Big\}.
					\end{align*}
					We check that $\xi$ satisfies \eqref{eqn:weakVarIneq2a}, i.e., $\xi$ is a desired subgradient.
			\end{itemize}
		\item[]\hspace*{-2em}\underline{To \eqref{eqn:weakRegEnergyInequality}:}
			Let $t_1$ and $t_2$ with $0\leq t_1\leq t_2\leq T$ be arbitrary.
			Integrating \eqref{eqn:discrEI} from Lemma \ref{lemma:discrEI} over the time interval $[t_1,t_2]$ yields
			\begin{align*}
 				\int_{t_1}^{t_2}\Big(\C F_\tau(t)+\C K_\tau(t)+\C D_\tau(0,t)+\C E_\tau(0,t)\Big)\dt
					\leq \int_{t_1}^{t_2}\Big(\C F(0)+\C K(0)+\C W_{ext}^\tau(0,t)\Big)\dt
			\end{align*}
			By the convergence properties in Lemma \ref{lemma:discrConvergence} and by lower semi-continuity arguments, we obtain
			\begin{align}
				\liminf_{\tau\searrow 0}\int_{t_1}^{t_2}\C F_\tau(t)\dt
					+\liminf_{\tau\searrow 0}\int_{t_1}^{t_2}\C K_\tau(t)\dt
				\geq{}& \int_{t_1}^{t_2}\Big(\C F(t)+\C K(t)\Big)\dt.
			\label{eqn:liminfFK}
			\end{align}
			The limit passage in the dissipation term $\int\C D_\tau(0,t)$ can be performed by Fatou's lemma, the estimate $t_\tau\geq t$,
			the weak convergence $\partial_t\widehat z_\tau\to \partial_{t}z$ in $L^2(0,T;L^2(\Omega))$ (see Lemma \ref{lemma:discrConvergence})
			and by a lower semi-continuity argument:
			\begin{align}
				\liminf_{\tau\searrow 0}\int_{t_1}^{t_2}\C D_\tau(0,t)\dt
				&\geq\liminf_{\tau\searrow 0}\int_{t_1}^{t_2}\int_0^{t_\tau}\int_\Omega|\partial_t\widehat z_\tau(s)|^2\dxs\dt\notag\\
				&\geq\liminf_{\tau\searrow 0}\int_{t_1}^{t_2}\int_0^{t}\int_\Omega|\partial_t\widehat z_\tau(s)|^2\dxs\dt\notag\\
					&\geq\int_{t_1}^{t_2}\bl\liminf_{\tau\searrow 0}\int_0^{t}\int_\Omega|\partial_t\widehat z_\tau(s)|^2\dxs\br\dt\notag\\
				&\geq \int_{t_1}^{t_2}\int_0^{t}\int_\Omega|\partial_{t}z(s)|^2\dxs\dt.
			\label{eqn:liminfD}
			\end{align}
			Moreover, Lemma \ref{lemma:discrConvergence} and Lebegue's convergence theorem lead to
			\begin{align}
				\lim_{\tau\searrow 0}\int_{t_1}^{t_2}\C W_{ext}^\tau(0,t)=0.
			\label{eqn:liminfW}
			\end{align}
			
			To treat the error term $\int\C E_\tau(0,t)$, we define $\C E_\tau(0,t)=:\C E_\tau^1(0,t)+\C E_\tau^2(0,t)$ with
			\begin{align*}
				\C E_\tau^1(0,t):={}&\int_0^{t_\tau}\int_\Omega\frac 12\frac{h(z_\tau^-)-h(z_\tau)}{\tau}\mathbf C\e(u_\tau^-):\e(u_\tau^-)\dxs
					+\int_0^{t_\tau}\int_\Omega W_{,z}(\e(u_\tau),z_\tau)\,\partial_t \widehat z_\tau\dxs\\
				\C E_\tau^2(0,t):={}&-\int_0^{t_\tau}\int_\Omega\frac{f(z_\tau)-f(z_\tau^-)}{\tau}\dxs
					+\int_0^{t_\tau}\int_\Omega f'(z_\tau)\,\partial_t \widehat z_\tau\dxs.
			\end{align*}
			By the differentiability of $h$, it holds
			$$
				h(z_\tau^-)=h(z_\tau)+h'(z_\tau)(z_\tau^--z_\tau)+r(z_\tau^--z_\tau),\;\frac{r(\eta)}{\eta}\to 0\text{ as }\eta\to 0.					
			$$
			We then get 
			\begin{align}
				&\int_0^{t_\tau}\int_\Omega\frac 12\frac{h(z_\tau^-)-h(z_\tau)}{\tau}\mathbf C\e(u_\tau^-):\e(u_\tau^-)\dxs\notag\\
				&\qquad=\int_0^{t_\tau}\int_{\{z_\tau^-(s)\neq z_\tau(s)\}}\frac 12\bl h'(z_\tau)\frac{z_\tau^--z_\tau}{\tau}
					+\frac{r(z_\tau^--z_\tau)}{z_\tau^--z_\tau}\frac{z_\tau^--z_\tau}{\tau}\br
					\mathbf C\e(u_\tau^-):\e(u_\tau^-)\dxs\notag\\
				&\qquad=
					\underbrace{\int_0^{t_\tau}\int_{\Omega}\frac 12 h'(z_\tau)\frac{z_\tau^--z_\tau}{\tau}\mathbf C\e(u_\tau^-):\e(u_\tau^-)\dxs}_{=:T_1}\notag\\
					&\qquad\quad+\underbrace{\int_0^{t_\tau}\int_{\{z_\tau^-(s)\neq z_\tau(s)\}}\frac 12\frac{r(z_\tau^--z_\tau)}{z_\tau^--z_\tau}\frac{z_\tau^--z_\tau}{\tau}
					\mathbf C\e(u_\tau^-):\e(u_\tau^-)\dxs}_{=:T_2}
			\label{eqn:energyTerm}
			\end{align}
			By using the convergence properties in Lemma \ref{lemma:discrConvergence},
			we find
			\begin{align*}
				T_1\to{}&\int_0^{t}\int_{\Omega}\frac 12 h'(z)\mathbf C\e(u):\e(u)\partial_{t}z\dxs\\
					&=\int_0^t\int_{\Omega}W_{,z}(\e(u),z)\partial_{t}z\dxs
			\end{align*}
			as $\tau\searrow 0$.
			
			The Lipschitz continuity of $h$ on the interval $[0,1]$ implies the boundedness of
			\begin{align*}
				\left\|\frac{r(z_\tau^--z_\tau)}{z_\tau^--z_\tau}\right\|_{L^\infty(\{z_\tau^-\neq z_\tau\})}
				&\leq\left\|\frac{h(z_\tau^-)-h(z_\tau)}{z_\tau^--z_\tau}\right\|_{L^\infty(\{z_\tau^-\neq z_\tau\})}
					+\left\|h'(z_\tau)\frac{z_\tau^--z_\tau}{z_\tau^--z_\tau}\right\|_{L^\infty(\{z_\tau^-\neq z_\tau\})}\\
				&\leq C.
			\end{align*}
			Taking also $\frac{r(z_\tau^--z_\tau)}{|z_\tau^--z_\tau|}\to 0$ a.e. in $\Omega\times(0,T)$ as $\tau\searrow 0$
			into account, we conclude by Lebesgue's generalized convergence theorem
			\begin{align}
				&\left\|\frac{r(z_\tau^--z_\tau)}{z_\tau^--z_\tau}\right\|_{L^q(\{z_\tau^-\neq z_\tau\})}\to0\text{ for every }q\in[1,\infty).
			\label{eqn:rTerm}
			\end{align}
			Therefore, we find by H\"older's inequality
			\begin{align*}
				T_2\leq{}&\frac12\int_0^{t_\tau}\Big\|\frac{r(z_\tau^-(s)-z_\tau(s))}{z_\tau^-(s)-z_\tau(s)}\Big\|_{L^4(\{z_\tau^-(s)\neq z_\tau(s)\})}
					\|\partial_t\widehat z_\tau(s)\|_{L^2(\Omega)}\|\e(u_\tau^-(s))\|_{L^4(\Omega)}\ds\\
				\leq{}&C\Big\|\frac{r(z_\tau^--z_\tau)}{z_\tau^--z_\tau}\Big\|_{L^4(\{z_\tau^-\neq z_\tau\})}
					\|\partial_t\widehat z_\tau\|_{L^2(0,T;L^2(\Omega))}\|u_\tau^-\|_{L^\infty(0,T;H^2(\Omega;\Rn))}
			\end{align*}
			Since $\big\|\frac{r(z_\tau^--z_\tau)}{z_\tau^--z_\tau}\big\|_{L^4(\{z_\tau^-\neq z_\tau\})}\to 0$
			by \eqref{eqn:rTerm} and $\|\partial_t\widehat z_\tau\|_{L^2(0,T;L^2(\Omega))}$ as well as $\|u_\tau^-\|_{L^\infty(0,T;H^2(\Omega;\Rn))}$
			are bounded by Lemma \ref{lemma:aprioriDiscr}, we obtain $T_2\to 0$ as $\tau\searrow 0$.
			The convergence properties in Lemma \ref{lemma:discrConvergence} also yield
			$$
				\int_0^{t_\tau}\int_\Omega W_{,z}(\e(u_\tau),z_\tau)\,\partial_t \widehat z_\tau\dxs
					\to \int_0^{t}\int_\Omega W_{,z}(\e(u),z)\,\partial_{t}z\dxs.
			$$
			In particular, we have used $\e(u_\tau)\to \e(u)$ in $L^4(0,T;L^4(\Omega;\Rnn))$
			due to the $H^2$-regularization.
			
			Together with \eqref{eqn:energyTerm} and the identified limits for the terms $T_1$ and $T_2$, prove
			$\C E_\tau^1(0,t)\to 0$ as $\tau\searrow 0$ for all $t\in[0,T]$.
			Taking also the uniform boundedness $|\C E_\tau^1(0,t)|\leq C$ with respect to $\tau>0$ and $t\in[0,T]$
			into account (which follows from the a priori estimates in Lemma \ref{lemma:aprioriDiscr}), we find
			\begin{align}
				\lim_{\tau\searrow 0}\int_{t_1}^{t_2}\C E_\tau^1(0,t)\dt= 0
			\label{eqn:liminfE1}
			\end{align}
			by Lebesgue's convergence theorem.
			Note that the limit
			\begin{align}
				\lim_{\tau\searrow 0}\int_{t_1}^{t_2}\C E_\tau^2(0,t)\dt=0
			\label{eqn:liminfE2}
			\end{align}
			can be shown with the same arguments.

				Integrating \eqref{eqn:discrEI} over the time interval $[t_1,t_2]$ and applying $\liminf_{\tau\searrow 0}$ on both sides, we obtain
				\begin{align*}
					&\liminf_{\tau\searrow 0}\int_{t_1}^{t_2}\C F_\tau(t)\dt+\liminf_{\tau\searrow 0}\int_{t_1}^{t_2}\C K_\tau(t)\dt+\liminf_{\tau\searrow 0}\int_{t_1}^{t_2}\C D_\tau(0,t)\dt
						+\liminf_{\tau\searrow 0}\int_{t_1}^{t_2}\C E_\tau(0,t)\dt\\
					&\qquad\leq \int_{t_1}^{t_2}\Big(\C F(0)+\C K(0)\Big)\dt+\liminf_{\tau\searrow 0}\int_{t_1}^{t_2}\C W_{ext}^\tau(0,t)\dt.
				\end{align*}
				The convergence properties \eqref{eqn:liminfFK}, \eqref{eqn:liminfD}, \eqref{eqn:liminfW}, \eqref{eqn:liminfE1} and \eqref{eqn:liminfE2}
				show
				\begin{align*}
					&\int_{t_1}^{t_2}\Big(\C F(t)+\C K(t)+\C D_\tau(0,t)\Big)\dt
						\leq \int_{t_1}^{t_2}\Big(\C F(0)+\C K(0)+\C W_{ext}(0,t)\Big)\dt.
				\end{align*}
				Since $t_1$ and $t_2$ with $0\leq t_1\leq t_2\leq T$ are arbitrary, we obtain \eqref{eqn:weakRegEnergyInequality}.
	\end{itemize}
	Hence, we have established existence of weak solutions in the sense of Definition \ref{def:regWeakSolution}.\ep
	
\subsection{Existence of weak solutions for the limit system}
\label{section:limit}

	In this section, we are going to prove Theorem \ref{theorem:mainResult}
	by performing a limit analysis $\delta\searrow 0$ for solution $(u_\delta,z_\delta,\xi_\delta)$
	from Proposition \ref{proposition:regSystem}.
	To this end, we approximate the data $(u^0,\ell,b)$ given in Theorem \ref{theorem:mainResult} by smooth functions (e.g. via convolution)
	$u_\delta^0\in H^2(\Omega;\Rn)$, $\ell_\delta\in C^{0,1}(0,T;L^2(\Omega;\Rn))$, $b_\delta\in C^{2,1}(0,T;H^2(\Omega;\Rn))$ such that
	\begin{subequations}
	\begin{align}
	\label{eqn:convUZeroDelta}
		u_\delta^0&\to u^0&&\text{strongly in }H^1(\Omega;\R^n),\\
	\label{eqn:convEllDelta}
		\ell_\delta&\to\ell&&\text{strongly in }L^2(0,T;L^2(\Omega;\R^n)),\\
	\label{eqn:convBDelta}
		b_\delta&\to b&&\text{strongly in }H^1(0,T;H^1(\Omega;\R^n)\cap H^2(0,T;L^2(\Omega;\R^n))
	\end{align}
	\end{subequations}
	as $\delta\searrow 0$.
	By possibly reparametrizing $\{u_\delta^0\}$, we obtain the following a priori estimate uniformly in $\delta$:
	\begin{align}
		\sqrt{\delta}\|u_\delta^0\|_{H^1}\leq C.
	\label{eqn:uZeroDeltaAPriori}
	\end{align}

	The cornerstone of the passage $\delta\searrow 0$ in the weak formulation are the following a priori estimates for $u_\delta$ and $z_\delta$ uniformly in $\delta$
	which are obtained by means of the regularized total energy-dissipation inequality
	\eqref{eqn:weakRegEnergyInequality}.

	
	\begin{lemma}[A priori estimates]
	\label{lemma:aprioriLimit}
		The following a priori estimates hold uniformly in $\delta$:
		\begin{subequations}
		\begin{align}
				\label{eqn:aprioriLimitU1}
					&\sqrt{\delta}\|u_\delta\|_{L^\infty(0,T;H^2(\Omega;\R^n))}\leq C,\\
				\label{eqn:aprioriLimitU2}
					&\|u_\delta\|_{L^\infty(0,T;H^1(\Omega;\R^n))\cap W^{1,\infty}(0,T;L^2(\Omega;\R^n))\cap H^2(0,T;(H_{\Gamma_\mathrm{D}}^2(\Omega;\R^n))^*)}\leq C,\\
				\label{eqn:aprioriLimitZ}
					&\|z_\delta\|_{L^\infty(0,T;\Wa)\cap H^1(0,T;L^2(\Omega))}\leq C.
		\end{align}
		\end{subequations}
	\end{lemma}
	\begin{proof}
		By the a priori estimates in Lemma \ref{lemma:aprioriDiscr} (i) which are independent of $\tau$ and $\delta$, by lower semi-continuity of the norm
		and by \eqref{eqn:uZeroDeltaAPriori}, we find \eqref{eqn:aprioriLimitU1} and \eqref{eqn:aprioriLimitU2}.
		
		To gain the a priori estimate \eqref{eqn:aprioriLimitZ}, we use the total
		energy-dissipation inequality \eqref{eqn:weakRegEnergyInequality}.
		It remains to estimate the following terms occurring in the $\C W_{ext}^\delta$-term:
		\begin{align*}
			&\int_0^t\int_\Omega W_{,e}(\e(u_\delta),z_\delta):\e(\partial_{t}b)\dxs
				\leq \|u_\delta\|_{L^2(0,T;H^1(\Omega;\Rn))}\|b\|_{H^1(0,T;H^1(\Omega;\Rn))},\\
			&-\int_0^t\int_{\Omega}\partial_t u_\delta\cdot \partial_{tt}b\dxs
				\leq \|u_\delta\|_{H^1(0,T;L^2(\Omega;\Rn))}\|b\|_{H^2(0,T;L^2(\Omega;\Rn))},\\
			&\int_\Omega \partial_t u_\delta(t)\cdot \partial_{t}b(t)\dx
				\leq \|u_\delta\|_{W^{1,\infty}(0,T;L^2(\Omega;\Rn))}\|b\|_{W^{1,\infty}(0,T;L^2(\Omega;\Rn))},\\
			&\int_0^t\int_{\Omega}\ell_\delta\cdot \partial_t(u_\delta-b)\dxs
				\leq \|\ell_\delta\|_{L^2(0,T;L^2(\Omega;\Rn))}\|u_\delta-b\|_{H^1(0,T;L^2(\Omega;\Rn))},\\
			&\delta\int_0^t \langle Au_\delta(t),\partial_t b(t)\rangle_{H^2}\dt
				\leq \delta\|u_\delta(t)\|_{L^\infty(0,T;H^2(\Omega;\Rn))}\|\partial_{t}b\|_{L^\infty(0,T;H^2(\Omega;\Rn))}.
		\end{align*}
		Taking the estimates \eqref{eqn:aprioriLimitU1} and \eqref{eqn:aprioriLimitU2} into account, we see that
		$\C W_{ext}^\delta(0,t)$ is uniformly bounded in $t$ and $\delta$.
		Therefore, by \eqref{eqn:weakRegEnergyInequality}, $\C F_\delta(t)$ is also uniformly bounded which proves \eqref{eqn:aprioriLimitZ}.
		\ep
	\end{proof}
	\begin{lemma}
	\label{lemma:convergenceLimit}
		There exist functions
		\begin{align*}
			&u\in L^\infty(0,T;H^1(\Omega;\R^n))\cap W^{1,\infty}(0,T;L^2(\Omega;\R^n))\cap H^2(0,T;(H_{\Gamma_\mathrm{D}}^2(\Omega;\R^n))^*),\hspace*{3.5em}\\
			&z\in L^\infty(0,T;\Wa)\cap H^1(0,T;L^2(\Omega))
		\end{align*}
		satisfying \eqref{eqn:initialBoundary}, \eqref{eqn:weakVarIneq1b}, \eqref{eqn:weakVarIneq2b}
		and a subsequence $\{\delta_k\}_{k\in\N}$ with $\delta_k\searrow 0$ as $k\nearrow\infty$ such that
		\begin{subequations}
		\begin{align}
			&u_{\delta_k}\to u&&\text{weakly-star in }L^\infty(0,T;H^1(\Omega;\R^n))\cap W^{1,\infty}(0,T;L^2(\Omega;\R^n))\notag\\
		\label{eqn:convLimitU1}
			&&&\hspace*{6.3em}\text{and weakly in }H^{2}(0,T;(H_{\Gamma_\mathrm{D}}^2(\Omega;\R^n))^*),\\
		\label{eqn:convLimitZ1}
			&z_{\delta_k}\to z&&\text{weakly-star in } L^\infty(0,T;\Wa)\text{ and weakly in }H^1(0,T;L^2(\Omega)),\\
		\label{eqn:convLimitZ2}
			&z_{\delta_k}\to z&&\text{strongly in } L^p(0,T;W^{1,p}(\Omega)),\\
		\label{eqn:convLimitZ3}
			&z_{\delta_k}\to z&&\text{strongly in } C(\ol{\Omega}\times[0,T])
		\end{align}
		\end{subequations}
		as $k\nearrow\infty$.
	\end{lemma}
	\begin{proof}
		Properties \eqref{eqn:convLimitU1}, \eqref{eqn:convLimitZ1} and \eqref{eqn:convLimitZ3}
		for some functions $u$ and $z$ follow with standard compactness
		and Aubin-Lions type results (cf. \cite{Simon}) by keeping $p\in(n,\infty)$ in mind.
		
		We also obtain by Aubin-Lions
		\begin{align}
		\label{eqn:convLimitZ4}
			z_{\delta_k}\to z\quad\text{ strongly in }L^q(0,T;L^\infty(\Omega))\text{ for every }q\in[1,\infty)
		\end{align}
		and thus (by possibly resorting to a further subsequence)
		\begin{align}
		\label{eqn:convLimitZ5}
			z_{\delta_k}(t)\to z(t)\quad\text{ weakly in }W^{1,p}(\Omega)
		\end{align}
		as $k\nearrow\infty$ and for a.e. $t\in(0,T)$.
		
		The strong convergence \eqref{eqn:convLimitZ2} can be shown as in Lemma \ref{lemma:strongConvZ}:
		We apply Lemma \ref{lemma:approximation}. Due to \eqref{eqn:convLimitZ5}, one finds an approximation sequence
		$\{\zeta_{\delta_k}\}\subseteq L^\infty(0,T;W_+^{1,p}(\Omega))$
		satisfying $\zeta_{\delta_k}\to\zeta$ in $L^p(0,T;W_+^{1,p}(\Omega))$ and $0\leq \zeta_{\delta_k}\leq z_{\delta_k}$ a.e. in $\Omega\times(0,T)$.
		By using the same $p$-monotonicity estimate as in \eqref{eqn:zEst},
		testing \eqref {eqn:weakDamageLaw2} with $\zeta_{\delta_k}-z_{\delta_k}$,
		using the a priori estimate \eqref{eqn:aprioriLimitU2} in Lemma \ref{lemma:aprioriLimit}
		and the convergence \eqref{eqn:convLimitZ4}, we obtain \eqref{eqn:convLimitZ2}.
		\ep
	\end{proof}
	
	As usual, we omit the subscript $k$.
	\begin{remark}
		Note that by \eqref{eqn:xiDef} and \eqref{eqn:aprioriLimitU2}, we only obtain an $L^\infty(0,T;L^1(\Omega))$-bound
		for $\xi_\delta$.
		A major challenge in the passage $\delta\searrow 0$ is
		to establish a desired subgradient $\xi\in L^1(\Omega\times(0,T))$ for the limit system.
	\end{remark}
	
	\textbf{Proof of Theorem \ref{theorem:mainResult}}
	
	We are going to prove that the functions $u$ and $z$ from Lemma \ref{lemma:convergenceLimit}
	also satisfy \eqref{eqn:weakMomentumBalance}, \eqref{eqn:weakVarIneq2a}, \eqref{eqn:weakDamageLaw2}
	and \eqref{eqn:weakEnergyInequality}.
	\begin{itemize}
		\item[]\hspace*{-2em}\underline{To \eqref{eqn:weakMomentumBalance}:}
			Integrating \eqref{eqn:weakRegMomentumBalance} over time from $0$ to $T$ and using the definition for the elastic energy density in \eqref{eqn:defW},
			we find
			\begin{align*}
				&\int_0^T\langle \partial_{tt} u_\delta(t),\zeta(t)\rangle_{H_\GammaD^2}\dt
					+\int_0^T\int_\Omega h(z_\delta)\CC\e(u_\delta):\e(\zeta)\dx+\delta\int_0^T\int_\Omega\langle\nabla(\nabla u_\delta),\nabla(\nabla\zeta)\rangle\dxt\\
					&\qquad=\int_0^T\int_\Omega \ell_\delta\cdot\zeta\dx
			\end{align*}
			for all $\zeta\in L^2(0,T;H_\GammaD^2(\Omega;\Rn))$.
			By exploiting the convergences \eqref{eqn:convEllDelta}, \eqref{eqn:convLimitU1}, \eqref{eqn:convLimitZ3} and
			$$
				\delta\Big|\int_0^T\int_\Omega\langle\nabla(\nabla u_\delta),\nabla(\nabla\zeta)\rangle\dxt\Big|
				\leq \delta\|u_\delta\|_{L^2(0,T;H^2(\Omega;\Rn))}\|\zeta\|_{L^2(0,T;H^2(\Omega;\Rn))}\to 0
			$$
			due to \eqref{eqn:aprioriLimitU1}, we conclude \eqref{eqn:weakMomentumBalance}
			for all $\zeta\in H_{\Gamma_\mathrm{D}}^2(\Omega;\R^n)$ and a.e. $t\in(0,T)$.
			
			By using the density of the set $H_{\Gamma_\mathrm{D}}^2(\Omega;\R^n)$ in $H_{\Gamma_\mathrm{D}}^1(\Omega;\R^n)$ (here we need the assumption that the boundary parts
			$\Gamma_\mathrm{D}$ and $\Gamma_\mathrm{N}$ have finitely many path-connected components, see \cite{Ber11}),
			we identify
			$\partial_{tt}u(t)\in (H_{\Gamma_\mathrm{D}}^1(\Omega;\R^n))^*$.
			Consequently, the equation \eqref{eqn:weakMomentumBalance} is true for all $\zeta\in H_{\Gamma_\mathrm{D}}^1(\Omega;\R^n)$ and a.e. $t\in(0,T)$.
			In particular, $\partial_{tt}u\in L^\infty(0,T;(H_{\Gamma_\mathrm{D}}^1(\Omega;\R^n))^*)$.
		\item[]\hspace*{-2em}\underline{To \eqref{eqn:weakVarIneq2a} and \eqref{eqn:weakDamageLaw2}:}
			We choose the following cluster points
			\begin{align}
				\chi_\delta:=\chi_{\{z_\delta>0\}}&\to\chi&&\text{weakly-star in }L^\infty(\Omega_T),\\
				\eta_\delta:=\chi_{\{z_\delta=0\}\cap\{W_{,z}(\e(u_\delta),z_\delta)+f'(z_\delta)\leq 0\}}&\to\eta&&\text{weakly-star in }L^\infty(\Omega_T),\\
			\label{eqn:weakConv15}
				F_\delta:=\chi_{\{z_\delta>0\}}\e(u_\delta)&\to F&&\text{weakly in }L^2(\Omega_T;\R^{n\times n}),\\
				G_\delta:=\chi_{\{z_\delta=0\}\cap\{W_{,z}(\e(u_\delta),z_\delta)+f'(z_\delta)\leq 0\}}\e(u_\delta)
					&\to G&&\text{weakly in }L^2(\Omega_T;\R^{n\times n})
			\end{align}
			as $\delta\searrow 0$ for a subsequence.
			By \eqref{eqn:convLimitZ3} and \eqref{eqn:convLimitU1}, we
			obtain for a.e. $x\in\{z>0\}$ 
			\begin{align}
			\label{eqn:limitProp}
				\chi(x)=1,\;\eta(x)=0,\;F(x)=\e(u)(x),\;G(x)=0
			\end{align}
			by the following comparison argument:
			
			Let $\zeta\in L^2(\Omega_T;\R^{n\times n})$ with $\mathrm{supp}(\zeta)\subseteq\{z>0\}$.
			Then, by \eqref{eqn:convLimitZ3}, we find $\mathrm{supp}(\zeta)\subseteq\{z_\delta>0\}$ for all sufficiently small $\delta>0$.
			On the one hand, \eqref{eqn:weakConv15} implies
			$$
				\int_0^T\int_{\Omega}F_\delta:\zeta\dxt\to\int_0^T\int_{\Omega}F:\zeta\dxt.
			$$
			On the other hand, by $\mathrm{supp}(\zeta)\subseteq\{z_\delta>0\}$ and \eqref{eqn:convLimitU1}
			$$
				\int_0^T\int_{\Omega}F_\delta:\zeta\dxt=\int_0^T\int_{\Omega}\e(u_\delta):\zeta\dxt\to\int_0^T\int_{\Omega}\e(u):\zeta\dxt.
			$$
			Thus $\int_0^T\int_{\Omega}\e(u):\zeta\dxt=\int_0^T\int_{\Omega}F:\zeta\dxt$ and, consequently, $F=\e(u)$ a.e. in $\{z>0\}$.
			The other identities in \eqref{eqn:limitProp} follow analogously.
			
			Now, let $\zeta\in L^\infty(0,T;W_-^{1,p}(\Omega))$.
			Taking \eqref{eqn:xiDef} into account, inequality \eqref{eqn:weakDamageLaw2} becomes after integration over time
			\begin{align}
				0\leq{}&\int_0^T\int_{\Omega}\bl|\nabla z_\delta|^{p-2}\nabla z_\delta\cdot\nabla\zeta+\partial_t z_\delta\zeta\br\dxt
				+\int_{\{z_\delta>0\}}\bl W_{,z}(\e(u_\delta),z_\delta)+f'(z_\delta)\br\zeta\dxt\notag\\
				&+\int_{\{z_\delta=0\}\cap \{W_{,z}(\e(u_\delta),z_\delta)+f'(z_\delta)\leq 0\}}\bl W_{,z}(\e(u_\delta),z_\delta)+f'(z_\delta)\br\zeta\dxt
			\label{eqn:cancelTrick}
			\end{align}
			for all $\zeta\in L^\infty(0,T;W_-^{1,p}(\Omega))$.
			Applying $\limsup_{\delta\searrow 0}$ on both sides and multiplying by $-1$ yield
			\begin{align*}
				0\geq{}&\lim_{\delta\searrow 0}\int_0^T\int_{\Omega}\bl|\nabla z_\delta|^{p-2}\nabla z_\delta\cdot\nabla(-\zeta)+\partial_t z_\delta(-\zeta)\br\dxt\\
				&+\liminf_{\delta\searrow 0}\int_0^T\int_{\Omega}h'(z_\delta)\CC F_\delta:F_\delta(-\zeta)\dxt+\lim_{\delta\searrow 0}\int_0^T\int_{\Omega}\chi_\delta\,f'(z_\delta)(-\zeta)\dxt\\
				&+\liminf_{\delta\searrow 0}\int_0^T\int_{\Omega}h'(z_\delta)\CC G_\delta:G_\delta(-\zeta)\dxt+\lim_{\delta\searrow 0}\int_0^T\int_{\Omega}\eta_\delta\,f'(z_\delta)(-\zeta)\dxt.
			\end{align*}
			Weakly lower semi-continuity arguments, the convergence property \eqref{eqn:convLimitZ3} and the identifications listed in \eqref{eqn:limitProp} give
			\begin{align*}
				0\geq{}&\int_0^T\int_{\Omega}\bl|\nabla z|^{p-2}\nabla z\cdot\nabla(-\zeta)+\partial_t z(-\zeta)\br\dxt\\
				&+\int_{\{z>0\}}\bl W_{,z}(\e(u),z)+f'(z)\br(-\zeta)\dxt\\
				&+\int_{\{z=0\}}\Big(h'(z)(\CC F:F+\CC G:G)+(\chi+\eta) f'(z)\Big)(-\zeta)\dxt.
			\end{align*}
			This inequality may also be rewritten in the following form:
			\begin{align*}
				0\leq{}&\int_0^T\int_{\Omega}\Big(|\nabla z|^{p-2}\nabla z\cdot\nabla\zeta+\bl W_{,z}(\e(u),z)+f'(z)+\partial_t z\br\zeta\Big)\dxt\\
				&+\int_{\{z=0\}}\Big(h'(z)(\CC F:F+\CC G:G)+(\chi+\eta) f'(z)-W_{,z}(\e(u),z)-f'(z)\Big)\zeta\dxt.
			\end{align*}
			Therefore,
			\begin{align*}
				0\leq{}&\int_0^T\int_{\Omega}\Big(|\nabla z|^{p-2}\nabla z\cdot\nabla\zeta+\bl W_{,z}(\e(u),z)+f'(z)+\partial_t z+\xi\br\zeta\Big)\dxt
			\end{align*}
			with
			$$
				\xi:=\chi_{\{z=0\}}\mathrm{min}\Big\{0,h'(z)(\CC F:F+\CC G:G)+(\chi+\eta-1) f'(z)-W_{,z}(\e(u),z)\Big\}.
			$$
			We obtain \eqref{eqn:weakVarIneq2a} and \eqref{eqn:weakDamageLaw2}.
			
		\item[]\hspace*{-2em}\underline{To \eqref{eqn:weakEnergyInequality}:}
			In order to establish the energy inequality \eqref{eqn:weakEnergyInequality}, we can proceed as in the proof of Proposition \ref{proposition:regSystem}:
			
			In fact, integrating the regularized energy-dissipation inequality \eqref{eqn:weakRegEnergyInequality} over the time interval $[t_1,t_2]$ with
			$0\leq t_1\leq t_2\leq T$,
			using the estimate $\frac{\delta}{2}\langle Au_\delta(t),u_\delta(t)\rangle_{H^2}\geq 0$ on the left hand side
			and passing to the limit by using lower semi-continuity arguments, the convergences in Lemma \ref{lemma:convergenceLimit} and
			\eqref{eqn:convUZeroDelta}-\eqref{eqn:convBDelta} as well as Fatou's lemma yields 
			\begin{align*}
				\int_{t_1}^{t_2}\big(\C F(t)+\C K(t)+\C D(0,t)\big)\dt\leq
					\int_{t_1}^{t_2}\big(\C F(0)+\C K(0)+\C W_{ext}(0,t)\big)\dt.
			\end{align*}
			Since $0\leq t_1\leq t_2\leq T$ were arbitrary, \eqref{eqn:weakEnergyInequality} holds for a.e. $t\in(0,T)$.
			\ep
	\end{itemize}

\begin{scriptsize} 

\bibliographystyle{alpha}
\bibliography{references} 

\begin{thebibliography}{FKNS98}

\bibitem[Ber11]{Ber11}
J.-M.E. Bernard.
\newblock {Density results in {S}obolev spaces whose elements vanish on a part
  of the boundary.}
\newblock {\em Chin. Ann. Math., Ser. B}, 32(6):823--846, 2011.

\bibitem[BS04]{BS04}
E.~Bonetti and G.~Schimperna.
\newblock Local existence for {F}r\'emond's model of damage in elastic
  materials.
\newblock {\em Contin. Mech. Thermodyn.}, 16(4):319--335, 2004.

\bibitem[BSS05]{BSS05}
E.~Bonetti, G.~Schimperna, and A.~Segatti.
\newblock On a doubly nonlinear model for the evolution of damaging in
  viscoelastic materials.
\newblock {\em J. of Diff. Equations}, 218(1):91--116, 2005.

\bibitem[FKNS98]{FKNS98}
M.~{Fr\'emond}, K.L. {Kuttler}, B.~{Nedjar}, and M.~{Shillor}.
\newblock {One-dimensional models of damage.}
\newblock {\em {Adv. Math. Sci. Appl.}}, 8(2):541--570, 1998.

\bibitem[FKS99]{FKS99}
M.~{Fr\'emond}, K.L. {Kuttler}, and M.~{Shillor}.
\newblock {Existence and uniqueness of solutions for a dynamic one-dimensional
  damage model.}
\newblock {\em {J. Math. Anal. Appl.}}, 229(1):271--294, 1999.

\bibitem[FN96]{FN96}
M.~Fr{\'e}mond and B.~Nedjar.
\newblock Damage, gradient of damage and principle of virtual power.
\newblock {\em Int. J. Solids Structures}, 33(8):1083--1103, 1996.

\bibitem[Fr{\'e}12]{Fr12}
M.~Fr{\'e}mond.
\newblock {\em Phase Change in Mechanics}.
\newblock Lecture Notes of the Unione Matematica Italiana. Springer, 2012.

\bibitem[HK11]{WIAS1520}
C.~Heinemann and C.~Kraus.
\newblock Existence of weak solutions for {C}ahn-{H}illiard systems coupled
  with elasticity and damage.
\newblock {\em Adv. Math. Sci. Appl.}, 21(2):321--359, 2011.

\bibitem[HK13]{WIAS1569}
C.~Heinemann and C.~Kraus.
\newblock Existence results for diffuse interface models describing phase
  separation and damage.
\newblock {\em Eur. J. Appl. Math.}, 24(2):179--211, 2013.

\bibitem[KRZ13a]{KRZ11}
D.~{Knees}, R.~{Rossi}, and C.~{Zanini}.
\newblock {A vanishing viscosity approach to a rate-independent damage model.}
\newblock {\em {Math. Models Methods Appl. Sci.}}, 23(4):565--616, 2013.

\bibitem[KRZ13b]{KRZ13}
D.~Knees, R.~Rossi, and C.~Zanini.
\newblock A quasilinear differential inclusion for viscous and rate-independent
  damage systems in non-smooth domains.
\newblock {\em WIAS preprint 1867}, 2013.

\bibitem[MR06]{Mielke06}
A.~Mielke and T.~Roub{\'i}{\v c}ek.
\newblock Rate-independent damage processes in nonlinear elasticity.
\newblock {\em Mathematical Models and Methods in Applied Sciences},
  16:177--209, 2006.

\bibitem[MT10]{MT10}
A.~Mielke and M.~Thomas.
\newblock Damage of nonlinearly elastic materials at small strain ---
  {E}xistence and regularity results.
\newblock {\em ZAMM Z. Angew. Math. Mech}, 90:88--112, 2010.

\bibitem[RR14]{RR12}
E.~Rocca and R.~Rossi.
\newblock A degenerating {PDE} system for phase transitions and damage.
\newblock {\em Math. Models Methods Appl. Sci.}, 24:1265--1341, 2014.

\bibitem[Sim86]{Simon}
J.~Simon.
\newblock Compact sets in the space ${L}^p(0,{T};{B})$.
\newblock {\em Annali di Matematica Pura ed Applicata}, 146:65--96, 1986.

\end{thebibliography}

\end{scriptsize}

\end{document}